\documentclass{amsart}

\usepackage{graphicx}

\usepackage{xcolor}

\usepackage{amsmath, amssymb, amsfonts, amsthm, fullpage}
\usepackage{graphics,graphicx}
\usepackage{accents}
\usepackage{color}

\usepackage{float} 

\renewcommand{\div}{\mathrm{div}}
\newcommand{\curl}{\mathrm{curl}}

\newcommand{\R}{\mathbb{R}}

\newcommand{\norma}[1]{{\left\vert\kern-0.25ex\left\vert\kern-0.25ex\left\vert #1 
    \right\vert\kern-0.25ex\right\vert\kern-0.25ex\right\vert}}

\newcommand{\be}{\mathbf{e}}    
\newcommand{\bx}{\mathbf{x}}
\newcommand{\bn}{\mathbf{n}}
\newcommand{\bu}{\mathbf{u}}
\newcommand{\bw}{\mathbf{w}}
\newcommand{\bz}{\mathbf{z}}
\newcommand{\bv}{\mathbf{v}}

\newcommand{\bH}{\mathbf{H}}

\newcommand{\bR}{\mathbf{R}}
\newcommand{\bS}{\mathbf{S}}
\newcommand{\bL}{\mathbf{L}}
\newcommand{\bI}{\mathbf{I}}
\newcommand{\bphi}{\boldsymbol{\phi}}
\newcommand{\bchi}{\boldsymbol{\chi}}
\newcommand{\bpsi}{\boldsymbol{\psi}}
\newcommand{\bzeta}{\boldsymbol{\zeta}}
\newcommand{\btau}{\boldsymbol{\tau}}
\newcommand{\bxi}{\boldsymbol{\xi}}

\newcommand{\Div}{\nabla\!\cdot\!}
\newcommand{\Curl}{\nabla\!\times\!}

\newcommand{\tbn}[1]{{\left\vert\kern-0.25ex\left\vert\kern-0.25ex\left\vert #1 \right\vert\kern-0.25ex\right\vert\kern-0.25ex\right\vert}}

\newtheorem{remark}{Remark}[section]
\newtheorem{lemma}{Lemma}[section]
\newtheorem{proposition}{Proposition}[section]
\newtheorem{theorem}{Theorem}[section]

\AtBeginDocument{%
  \setlength{\belowdisplayskip}{5pt} \setlength{\belowdisplayshortskip}{5pt}
  \setlength{\abovedisplayskip}{5pt} \setlength{\abovedisplayshortskip}{5pt}
}

\begin{document}

\title[A regularized shallow-water waves system with slip-wall boundary conditions in a basin]{A regularized shallow-water waves system with slip-wall boundary conditions in a basin: Theory and numerical analysis}

\author{Samer Israwi}
\address{\textbf{S.~Israwi:} Department of Mathematics, Faculty of Sciences 1,
Lebanese University, Beirut, Lebanon
}
\email{s\_israwi83@hotmail.com}

\author{Henrik Kalisch}
\address{\textbf{H.~Kalisch:} Department of Mathematics, University of Bergen, Postbox 7800, 5020 Bergen, Norway}
\email{henrik.kalisch@uib.no}

\author{Theodoros Katsaounis}
\address{\textbf{Th.~Katsaounis:}   Dept. of Math. and Applied Mathematics, Univ. of Crete, Heraklion, Greece \\ \& IACM, FORTH, Heraklion, Greece \\ \& CEMSE, KAUST, Thuwal, Saudi Arabia}
\email{theodoros.katsaounis@kaust.edu.sa}

\author{Dimitrios Mitsotakis}
\address{\textbf{D.~Mitsotakis:} Victoria University of Wellington, School of Mathematics and Statistics, PO Box 600, Wellington 6140, New Zealand}
\email{dimitrios.mitsotakis@vuw.ac.nz}




\date{\today}



\begin{abstract}
The simulation of long, nonlinear dispersive waves in bounded domains usually requires the use 
of slip-wall boundary conditions. Boussinesq systems appearing in the literature are generally
not well-posed when such boundary conditions are imposed, or if they are well-posed it is very 
cumbersome to implement the boundary conditions in numerical approximations. 

In the present paper a new Boussinesq system is proposed for the study of long waves of small amplitude 
in a basin when slip-wall boundary conditions are required. The new system is derived using asymptotic 
techniques under the assumption of small bathymetric variations, and
a mathematical proof of well-posedness for the new system is developed. 

The new system is also solved numerically using a Galerkin finite-element method,
where the boundary conditions are imposed with the help of Nitsche's method. 
Convergence of the numerical method is analyzed, and precise error estimates are provided. 
The method is then implemented, and the convergence is verified using numerical experiments. 
Numerical simulations for solitary waves shoaling on a plane slope are also presented. The results
are compared to experimental data, and excellent agreement is found.
\end{abstract}

\maketitle

\section{Introduction}
In this work, attention is given to a new model system for the study of long waves
of small amplitude at the free surface of a perfect fluid. The system can be used
in the presence of non-constant bathymetry and lateral boundaries. The main
new feature of the system is that it is straightforward to implement slip-wall
boundary conditions on a finite domain. The system falls in the general class
of Boussinesq systems which have become standard tools in the study of nearshore
hydrodynamics. 

While the full water-wave problem is described by the Euler equations
with free surface boundary conditions \cite{Whitham2011}, it is well known that this problem is difficult 
to treat both mathematically and numerically. 
In particular, it is not known whether solutions exist on relevant time scales,
and numerical simulations of the full water-wave problem 
may suffer from serious stability issues.
Therefore, in practical situations in coastal hydrodynamics, 
asymptotic approximations of the Euler equations are often used
to find simpler systems that describe the main features of the flow. 
These simplified systems are usually derived using the long-wave assumption. 
The simplest in structure of such long wave systems are the shallow-water wave equations 
(or Saint-Venant equations) which take the form 
\begin{equation}
\label{eq:sweqs}
\begin{aligned}
&\eta_t+\Div[(D+\eta)\bu]=0\ , \\
&{\bf u}_t+g\nabla\eta+(\bu\cdot \nabla)\bu=0\ , 
\end{aligned}
\end{equation}
where for the space variable $\bx=(x,y)\in \mathbb{R}^2$ and time $t\geq0$ the functions $\eta=\eta(\bx,t)$ and $\bu=\bu(\bx,t)$ denote the free surface elevation and  the depth-averaged horizontal 
velocity of the fluid, respectively. The function $D=D(\bx)>0$ represents the distance of the ocean floor from the undisturbed level of free-surface elevation, while in the previous notation $g$ is the gravitational acceleration constant. 
Since this system is hyperbolic, there is a number of well developed methods for 
the approximation of solutions such as TVD methods, Riemann solvers etc, \cite{leveque2002,toro2013},
and it is well known that the shallow-water system is most of the times able to describe
the propagation of tsunamis and flood waves.
It is also well known that smooth solutions of (\ref{eq:sweqs}) 
preserve the energy functional 
\begin{equation}\label{eq:firstenrg}
E(t)=\int g\eta^2+(D+\eta)|\bu|^2~ d\bx\ ,
\end{equation}
which is an approximation of the total energy satisfied by the solutions of the Euler equations. 
Although the shallow-water system has favorable properties and is widely used, 
it is restricted to the modeling of very long waves, and is not suitable for the description
of coastal phenomena such as solitary waves or periodic wave-trains.

In a seminal contribution, D.H. Peregrine in \cite{Per67} resolved this issue by deriving a 
Boussinesq-type system applicable to coastal wave phenomena such as shoaling solitary waves,
wave reflection and long-shore currents to name just a few.
The Peregrine system is written in dimensional form as
\begin{equation}
\label{eq:Peregrin}
\begin{aligned}
&\eta_t+\Div[(D+\eta)\bu]=0\ , \\
&{\bf u}_t+g\nabla\eta+(\bu\cdot \nabla)\bu -\frac{1}{2}D\nabla(\Div(D\bu_t))+\frac{1}{6}D^2\nabla(\Div\bu_t)=0\ ,
\end{aligned}
\end{equation}
and describes the propagation of water waves over a bottom topography 
$D=D(\bx)$ ($\bx=(x,y)\in \mathbb{R}^2$) with free surface elevation 
$\eta=\eta(\bx,t)$ and a depth-averaged horizontal velocity field $\bu=\bu(\bx,t)$. 

The first equation in Peregrine's system (\ref{eq:Peregrin}) is the exact expression 
of the mass conservation, and is derived from the kinematic free-surface boundary condition.
The second equation is derived from the dynamic boundary condition. 
Although Peregrine's system looks very convenient due to its simplicity, 
it appears to have several drawbacks in relation to existence and uniqueness of solutions
and numerical discretization. 
Indeed, it has only recently been proved that the Cauchy problem for the Peregrine system (\ref{eq:Peregrin}) 
is well-posed in $\mathbb{R}^2$ \cite{DuchIsr2018}, and it is still unknown whether 
the system is well-posed in bounded domains $\Omega\subset \mathbb{R}^2$. Note that solutions to Peregrine system do not satisfy any reasonable approximation to the total energy functional such as (\ref{eq:firstenrg}).
Moreover, as it was shown in \cite{KMS2019}, numerical discretization of 
Peregrine's system in bounded domains can yield suboptimal convergence rates 
and also low resolution phenomena (i.e. aliasing phenomena)
due to its hyperbolic form of the mass conservation.

Several other Boussinesq-type systems with certain favorable properties have been derived as alternatives to 
Peregrine's system. One such example, which is of central focus in the present paper, is the class of 
the BBM-BBM type systems. 
These systems were first introduced in \cite{BS1976,BC1998,BCS2002} 
in one dimension and later in \cite{BCL2015} in two dimensions, 
and they agree asymptotically with the Euler equations in the long-wave small-amplitude regime.
In particular in \cite{BCL2015} a BBM-BBM system of the form
\begin{equation} \label{eq:BBMsys}
\begin{aligned}
&\eta_t+\Div \bu +\Div(\eta \bu) -\frac{1}{6} \Delta \eta_t = 0 \ , \\
& \bu_t+g\nabla \eta+\frac{1}{2}\nabla |\bu|^2-\frac{1}{6} \Delta \bu_t= 0\ . 
\end{aligned}
\end{equation}
was derived in the case of a flat bottom, and a generalization of this system 
to the case of general topography was presented in \cite{mits2009}. Such systems can be used for the description of the generation and propagation of tsunamis, among other nonlinear and dispersive waves, \cite{Cumb1987,mits2009}. They are also very robust even in the presence of variable bottom topographic environments \cite{Senth2016}, (with some exceptions such as the KdV-KdV system which is asymptotically equivalent to the BBM-BBM but cannot serve its purpose due to physical inconsistencies, \cite{BDM2007i,BDM2008ii}).

The main characteristic of these systems is the presence of a dispersive term of mixed type,  
involving two space derivatives and one time derivative in both equations, as opposed to
the Peregrine system which features this term only in one of the two equations.
The idea of using mixed-derivative terms goes back to Peregrine \cite{Per66},
and the single KdV-type equation with a mixed-derivative term has become known as the
BBM equation \cite{BBM}. In the context of BBM-type systems, the inclusion of
the mixed-derivative term in the first equation has two drawbacks. First, the first equation is no
longer an exact mass conservation equation, and the mass balance now takes an approximate form
\cite{AlK2012}. However, mass is still conserved to within the order of approximation,
so this is not a serious problem. Secondly, the dispersion relation for the linearized
equation is slightly less accurate than the dispersion relation of the Peregrine system \cite{BCS2002}.  
This drawback can be mitigated by including higher-order dispersive terms which is the approach
followed in the present contribution. 

While the presence of the Laplace operator in the dispersive terms of the mass and momentum equations 
in this BBM-BBM system appears to be attractive from the point of view of mathematical analysis
and numerical discretization (such as well-posedness in the Hadamard sense and optimal convergence rates of numerical solutions), 
the initial value problem with wall boundary conditions for these kind of  systems  
in bounded domains requires Dirichlet boundary conditions for the velocity field 
on the boundary of the domain, \cite{DMS2009}, in addition to homogenous Neumann boundary conditions 
for the free surface $\eta$. The Dirichlet boundary condition for modeling walls needs to be zero in the direction of the unit normal vector of the boundary $\bu\cdot\bn=0$. On the other hand, the inversion of the operator $(I-\Delta)$ requires also information about the tangential component of the velocity on the wall $\bu\cdot\btau$, 
which is generally not available. For this reason,  this term is usually taken as $\bu\cdot\btau=0$ which results in overall zero Dirichlet boundary conditions $\bu=(\bu\cdot\bn)\bn+ (\bu\cdot\btau)\btau=0$.These boundary conditions are essentially no-slip wall boundary conditions and are quite restrictive, 
especially when one considers obstacles or other complicated boundaries of the numerical domain.

In order to address this problem, a new BBM-BBM type system suitable 
for slip-wall boundary conditions was recently proposed \cite{KMS2019}.
The system is written in dimensional variables as
\begin{equation}
\label{eq:nper3}
\begin{aligned}
& \eta_t+\Div((D+\eta)\bu)-\tilde{b}\Div[D^2(\nabla(\nabla D\cdot \bu)+\nabla D\Div \bu)]- (\tilde{a}+\tilde{b})\Div [D^2 \nabla \eta_t]=0\ , \\
& \bu_t+g\nabla\eta+ \frac{1}{2}\nabla|\bu|^2+[\tilde{c}D\nabla(\Div(D\bu_t))+\tilde{d}D^2\nabla(\Div\bu_t)]=0\ ,
\end{aligned}
\end{equation}
where
\begin{equation}\label{eq:coefsbous}
\tilde{a}=\theta-1/2,~ \tilde{b}=1/2[(\theta-1)^2-1/3],~ \tilde{c}=\theta-1~ \mbox{ and }~\tilde{d}=1/2(\theta-1)^2\ ,
\end{equation}
for $\theta\in [0,1]$. 
Here, $\bu$ denotes the horizontal velocity field at height $z=-D+\theta(\eta+D)$ above the bottom, 
instead of the depth-averaged horizontal velocity used in the Peregrine system (\ref{eq:Peregrin}). 
For $\theta=\sqrt{2/3}$ the BBM-BBM-type system of \cite{BCL2015} 
is recovered but with different dispersive terms in the second equation. 

By considering mild bottom topography in (\ref{eq:nper3}) with $\theta=\sqrt{2/3}$ one can obtain the system
\begin{equation}\label{newsys}
\begin{aligned}
& \eta_t+\Div((D+\eta)\bu) - \frac{1}{6} \Div (D^2 \nabla \eta_t)=0\ ,  \\
& \bu_t+g\nabla\eta+ \frac{1}{2}\nabla|\bu|^2-\frac{1}{6}\nabla(\Div(D^2\bu_t))=0\ . 
\end{aligned}
\end{equation}
The small bottom variations assumption on which this model is based was used before in \cite{Chen03,BL2009} for the derivation of simple equations with variable bottom topography. For a formal definition of the current in the presence of mild topography, see Section 3.2.
Note that the term $\Delta \bu_t$ of (\ref{eq:BBMsys}) has been replaced by the term $\nabla(\Div\ \bu_t)$ in the momentum equation of (\ref{newsys}). This new term allows the use of slip-wall boundary conditions when the problem is posed in bounded domains. The particular system, as we shall see later, appears to have certain advantages compared to other Boussinesq systems of water wave theory: (i) System (\ref{newsys}) is well-posed in bounded domains with slip-wall boundary conditions at least locally in time; (ii) has simple structure and preserves the same energy functional as its non-dispersive counterpart; (iii) its numerical discretization by Galerkin/Finite element method results in stable simulations.  These advantages are crucial since they are related to Newton's principle of determinacy of physically sound problem, a principle that all deterministic systems must obey. The present paper is devoted to the analysis of system (\ref{newsys}) and its Galerkin approximations. 

Specifically, we give a detailed explanation why the system (\ref{newsys})
is attractive for the study of shallow water waves.
For the derivation of the new system  we follow two different approaches: 
The first one is based on the classical asymptotic method taking as point of departure 
the full water-wave problem based on the Euler equations.  
In the derivation, we pay special attention to incorporate appropriate dispersive terms 
which yield the correct behavior in terms of energy conservation.  
As a consequence, the new system features energy conservation in a similar fashion as the Euler equations. 
In particular, the solutions of the new regularized system preserve the exact same energy 
as its non-dispersive counterpart, namely, the shallow-water waves system.  
Furthermore, we present an alternative derivation based on variational principles. 
This approach is quite attractive not only for its simplicity, 
but also for obtaining physical properties in a straightforward manner. 
Although the new system is derived with the assumption of the mild bottom topography, 
it will be shown in Section \ref{sec:numerics} 
that it is appeared to be valid even for more general bottom topographies.

Furthermore, we explore the theoretical background of 
(\ref{newsys})
insofar as it concerns the initial-boundary value 
problem in a bounded domain with slip-wall boundary conditions. 
These boundary conditions are necessary to describe water waves propagating 
in a closed basin, and in general to describe interactions of waves with solid walls. 
The initial-boundary value problem of the new system with slip-wall boundary conditions 
appears to have similar well-posedness properties with the classical BBM-BBM system studied in \cite{DMS2009}.

For numerical approximations we apply to the new system the Galerkin/Finite element method of \cite{KMS2019}. 
Due to the difficulty of incorporating the exact boundary conditions into the finite element space, 
we resort to applying the Nitsche method \cite{Nitsche}. This method is commonly used 
in practical problems but is rarely analyzed mathematically. Building on previous work in \cite{KMS2019}, 
we prove that the numerical solution converges to the exact solution.
These results are verified with actual numerical computations, and it is also shown
that at least in some cases the optimal rate of convergence is achieved.

The paper is organized as follows: 
First we present the derivation of the system using the two approaches in Section \ref{sec:derivation}. 
In Section \ref{sec:wellposed} we study the well-posedness 
of the specific initial-boundary value problem, 
a necessary ingredient for the justification of systems of modeling equations. 
The application of the finite element method for the discretization of the new system, 
its convergence and accuracy are presented in Section \ref{sec:numsols}. 
Finally, in Section \ref{sec:numerics} we consider several 
numerical experiments verifying the theoretical findings 
and demonstrating the applicability of the numerical method.

\section{Derivation of the new system}\label{sec:derivation}

In this section we present the derivation of the new system based on a classical asymptotic approach.  Furthermore,  we present a novel alternative derivation based on variational methods.

\subsection{Asymptotic reasoning}

In what follows we consider characteristic quantities for  typical waves in the Boussinesq regime, in particular a typical wave amplitude $a_0$ and length $\lambda_0$ and a typical constant depth $D_0$. We will denote the linear wave speed by $c_0=\sqrt{gD_0}$. The bottom topography is defined as $D=D_0+ D_b$ where $D_b$ is the typical deviation from the typical depth $D_0$. We also define the order of bottom topography variations $d_0$, and the dimensionless variables 
\begin{equation}
\tilde{\bf x}=\frac{{\bf x}}{\lambda_0},\qquad  \tilde{t}= \frac{c_0}{\lambda_0}t, \qquad \tilde{\bf u} = \frac{D_0}{a_0c_0}{\bf u}, \qquad \tilde{\eta}=\frac{\eta}{a_0}, \qquad \tilde{D}_b=\frac{D_b}{d_0}\ .
\end{equation} 
Then the BBM-BBM system \eqref{eq:nper3} can be written in the nondimensional 
and scaled form as
\begin{equation}
 \label{eq:prgnd}
\begin{aligned}
& \tilde{\eta}_t+\Div((1+\beta\tilde{D}_b+\varepsilon\tilde{\eta})\tilde{\bu}) -\sigma^2\tilde{b}\Div[(1+\beta \tilde{D_b})^2(\nabla(\nabla (1+\beta \tilde{D_b})\cdot \bu)+\nabla (1+\beta \tilde{D_b})\Div \tilde{\bu})] \\
&- \sigma^2(\tilde{a}+\tilde{b})\Div [(1+\beta \tilde{D_b})^2 \nabla \eta_t]=O(\varepsilon\sigma^2,\sigma^4)\ , \\
& \tilde{\bu}_t+\nabla\tilde{\eta}+ \varepsilon\frac{1}{2}\nabla|\tilde{\bu}|^2+\sigma^2[\tilde{c}(1+\beta \tilde{D_b})\nabla(\Div((1+\beta \tilde{D_b})\tilde{\bu}_t))+\tilde{d}(1+\beta \tilde{D_b})^2\nabla(\Div\tilde{\bu}_t)]=O(\varepsilon\sigma^2,\sigma^4)\ ,
\end{aligned}
\end{equation}
where the parameter $\varepsilon=\alpha_0/D_0$, $\sigma=D_0/\lambda_0$ and $\beta=d_0/D_0$ 
are all assumed to be positive and small: $0<\varepsilon,\sigma,\beta\ll 1 $.
The error terms appearing on the right hand side of the various equations comprise 
terms emerging from asymptotic expansions of the horizontal velocity.
Note that (\ref{eq:prgnd}) can be derived from the Euler system with the free surface boundary conditions 
in the same way as Peregrine's system (\ref{eq:Peregrin}) but by specifying the horizontal velocity of the fluid at certain depth
as the second dependent variable instead of the average horizontal velocity which is used
as dependent variable in the Peregrine system. Notice also that it is crucial to use the approximate 
irrotationality condition $\Curl \tilde{\bu}=O(\sigma^2)$, \cite{mits2009} in the derivation. 
This condition is also hidden behind the derivation of (\ref{eq:Peregrin}) 
and is a crucial component of the derivation of (\ref{eq:BBMsys}), 
indicating that irrotationality is an unavoidable component of Boussinesq systems.
Moreover, assuming that terms of $O(\beta\sigma^2)$ are negligible, \cite{Chen03}, 
the BBM-BBM system (\ref{eq:prgnd}) can be further simplified to 
\begin{equation}
 \label{eq:prgndlow}
\begin{aligned}
& \tilde{\eta}_t+\Div((1+\beta\tilde{D}_b+\varepsilon\tilde{\eta})\tilde{\bu}) - \sigma^2(\tilde{a}+\tilde{b})\Div ( \nabla \tilde{\eta}_t)=O(\varepsilon\sigma^2,\sigma^4,\beta\sigma^2)\ , \\
& \tilde{\bu}_t+\nabla\tilde{\eta}+ \varepsilon\frac{1}{2}\nabla|\tilde{\bu}|^2+\sigma^2(\tilde{c}+\tilde{d})\nabla(\Div\tilde{\bu}_t)=O(\varepsilon\sigma^2,\sigma^4,\beta\sigma^2)\ ,
\end{aligned}
\end{equation}
or in dimensional form, and after discarding the high-order terms
\begin{equation}
 \label{eq:prgndlowb}
\begin{aligned}
& \eta_t+\Div((D+\eta)\bu) - (\tilde{a}+\tilde{b})D_0^2\Div ( \nabla \eta_t)=0\ , \\
& \bu_t+g\nabla\eta + \frac{1}{2}\nabla|\bu|^2+(\tilde{c}+\tilde{d})D_0^2\nabla(\Div\bu_t)=0\ .
\end{aligned}
\end{equation}
We will refer to this system as simplified BBM-BBM system, which is a generalization 
of the analogous one-dimensional BBM-BBM system derived in \cite{Chen03}.
It is easily seen that the bottom variations practically do not contribute at all in the dispersive terms. 
As we shall see also later in Section \ref{sec:numerics}, such simplifications diminish the accuracy of the model
and make it inappropriate for practical applications such as the shoaling of solitary waves, 
even in the cases where the slope of the seafloor is mild.
On the other hand, keeping otherwise negligible high-order terms of $O(\beta\sigma^2)$ in the dispersive terms and taking the advantage of the fact that $\sigma^2(1+\beta\tilde{D}_b)\approx \sigma^2+O(\beta\sigma^2)$ to place the term $D$ at a position that ensures energy conservation, we obtain from (\ref{eq:prgnd}) the system
\begin{equation}
 \label{eq:prgnd2}
\begin{aligned}
& \tilde{\eta}_t+\Div((1+\beta\tilde{D}_b+\varepsilon\tilde{\eta})\tilde{\bu}) - \sigma^2(\tilde{a}+\tilde{b})\Div ((1+\beta \tilde{D_b})^2 \nabla \tilde{\eta}_t)=O(\varepsilon\sigma^2,\sigma^4,\beta\sigma^2)\ , \\
& \tilde{\bu}_t+\nabla\tilde{\eta}+ \varepsilon\frac{1}{2}\nabla|\tilde{\bu}|^2+\sigma^2(\tilde{c}+\tilde{d})(1+\beta \tilde{D_b})\nabla(\Div((1+\beta \tilde{D_b})\tilde{\bu}_t))=O(\varepsilon\sigma^2,\sigma^4,\beta\sigma^2)\ .
\end{aligned}
\end{equation}
As we shall see later, solutions of such a system can
preserve the same energy functional as the non-dispersive shallow water equations. 
Numerical experiments have shown that keeping topography variations in the high-order dispersive terms 
extends the validity of the model in practical problems such as the shoaling of 
long water waves over general bottoms.  Moreover, the model is more realistic since 
the actual bottom topography function $D$ appears in the equations 
instead of the typical depth $D_0$ (see e.g. \cite{Lannes13}).
The asymptotic equivalence of the equations with $D_0$ and $D$ 
enables us to reformulate them appropriately so that the resulting system will be Hamiltonian. 
For example, after neglecting the high-order terms and using dimensional variables 
the system (\ref{eq:prgnd2}) can be written as
\begin{equation}
 \label{eq:bbm}
\begin{aligned}
& \eta_t+\Div((D+\eta)\bu) - (\tilde{a}+\tilde{b}) \Div (D^2 \nabla \eta_t)=0\ , \\
& \bu_t+g\nabla\eta+ \frac{1}{2}\nabla|\bu|^2+(\tilde{c}+\tilde{d})\nabla(\Div(D^2\bu_t))=0\ .
\end{aligned}
\end{equation}
For the sake of completeness we present also the extension of (\ref{eq:bbm}) with moving bottom topography. Such systems can be useful in the studies of water waves generated by moving bottoms such as tsunamis \cite{DMCS12,DMGD13,Lannes13,mits2009}. Assuming moving bottom topographic features described 
by a bottom function of the form $D(\bx)+\zeta(\bx,t)$ where $\zeta$ has a 
typical magnitude of $O(a_0)$, the system (\ref{eq:bbm}) is written as
\begin{equation}
 \label{eq:bbmvb}
\begin{aligned}
& \eta_t+\Div((D+\zeta+\eta)\bu) - (\tilde{a}+\tilde{b})\Div (D^2 \nabla \eta_t)=-\tilde{a}\Div(D^2\nabla\zeta_t)-\zeta_t\ , \\
& \bu_t+g\nabla\eta+ \frac{1}{2}\nabla|\bu|^2+(\tilde{c}+\tilde{d})\nabla(\Div(D^2\bu_t))=-\tilde{c}D\nabla\zeta_{tt}\ .
\end{aligned}
\end{equation}

In this paper we will consider the system (\ref{eq:bbm}) 
in the case where $\theta=\sqrt{2/3}$ (i.e. $\tilde{a}+\tilde{b}=-\tilde{c}-\tilde{d}=1/6$) 
in a bounded domain $\Omega\subset \R^2$ with slip wall boundary conditions 
of the form $\nabla\eta\cdot \bn=0$ and $\bu\cdot \bn=0$ on the boundary $\partial\Omega$, 
where $\bn$ is the external unit normal vector to the boundary. 
We rewrite the BBM-BBM system (\ref{eq:bbm}) in the form of an initial-boundary value problem
\begin{equation}\label{eq:bbm2a}
\begin{aligned}
& \eta_t+\Div((D+\eta)\bu) - \frac{1}{6} \Div (D^2 \nabla \eta_t)=0\ ,  \\
& \bu_t+g\nabla\eta+ \frac{1}{2}\nabla|\bu|^2-\frac{1}{6}\nabla(\Div(D^2\bu_t))=0\ ,
\end{aligned}
\end{equation}
where the initial state of the problem is specified by the initial conditions 
\begin{equation}\label{eq:ics}
\eta(\bx,0)=\eta_0(\bx), \quad \bu(\bx,0)=\bu_0(\bx), \quad \forall \bx\in \Omega, 
\end{equation}
and on the boundary $\partial \Omega$ we assume physically important slip-wall boundary conditions 
\begin{equation}\label{eq:bcs}
\nabla\eta\cdot \bn=0,  \quad \bu\cdot \bn=0, \quad \text{ on } \partial\Omega\ .
\end{equation}
Compatibility boundary conditions on the initial data such as $\nabla\eta_0\cdot\bn=0$ and $\bu_0\cdot\bn=0$ should also be considered.
Equations (\ref{eq:bbm2a}), (\ref{eq:ics}), (\ref{eq:bcs}) form an initial-boundary value problem. Note that the Neumann boundary condition for $\eta$ is satisfied by the solutions of the Euler equations \cite{Khakimzyanov2018a}. This means that the particular boundary condition is physical and is not restrictive. On the other hand, systems like the Peregrine system, has been proved to be well-posed with boundary conditions only for the velocity and only in one-dimension \cite{adamy2011,FP2005}. This result doesn't guarantee though that the solution satisfies the additional requirement for the free surface, while the problem in 2D is still open and only experimental evidence exist (see for example \cite{KMS2019}).
\begin{remark}
We know from calculus that
$\nabla (\Div \bw)= \Delta \bw  + \Curl(\Curl \bw)$.
In our case, where $\Curl \bu_t=0$, we have that
$ \nabla (\Div \bu_t)= \Delta \bu_t\ .$ This implies that whenever the bottom is flat, 
the regularization operator $I-\frac{1}{6}\nabla (\Div ~)$ coincides with the classical elliptic operator $I-\frac{1}{6}\Delta$  and thus the theory of \cite{DMS2009} applies here too. In addition, using the small bottom variations assumption we conclude that this is still valid in the case of a variable bottom. Since the regularization properties of the aforementioned BBM-BBM system are expected to be the same as the original system of \cite{BCL2015}, we focus our attention to the new one due to its favorable properties when it comes to the application of the slip-wall boundary conditions.
\end{remark}

\subsection{Conservation properties and regularity}
Contrary to the classical BBM-BBM (and also Peregrine) type systems in 2D, the aforementioned BBM-BBM system is Hamiltonian. Specifically, any solution $(\eta,\bu)$ of the initial-boundary value problem  (\ref{eq:bbm2a})--(\ref{eq:bcs}) conserves the energy functional
\begin{equation}\label{eq:hamiltonian} 
E(t)\doteq \frac{1}{2}\int_{\Omega} g\eta^2+(D+\eta)|\bu|^2~d\bx\ ,
\end{equation}
in the sense that
$E(t)=E(0)$ for all $t>0$. The energy functional (\ref{eq:hamiltonian}) 
in non-dimensional variables takes the form
\begin{equation}\label{eq:hamiltonian2}
E_\varepsilon(t)\doteq \frac{1}{2}\int_{\Omega} \eta^2+(1+\beta \tilde{D}_b+\varepsilon\eta)|\bu|^2~d\bx\ . 
\end{equation}
The conservation of energy gives an upper bound of the $L^2$-norm of the solution. 
To show the conservation of energy we write system (\ref{eq:bbm2a}) in the form
\begin{equation}\label{eq:conservform}
\begin{aligned}
\eta_t+\nabla \cdot P =0\ ,\\
\bu_t+\nabla Q = 0\ ,
\end{aligned}
\end{equation}
where $P=(D+\eta)\bu-\frac{1}{6} D^2\nabla\eta_t$ and $Q=g\eta+\frac{1}{2}|\bu|^2-\frac{1}{6}\Div (D^2\bu_t)$. 
Then, after integrating by parts and applying the slip-wall boundary conditions at $\partial\Omega$ we have
\begin{align*}
0 &= \int_\Omega \eta_tQ+\bu_t\cdot P+\Div P Q+P\cdot\nabla Q~d\bx\\
&= \int_\Omega g\eta_t\eta+\frac{1}{2}\eta_t|\bu|^2-\frac{1}{6}\eta_t \Div (D^2\bu_t) +D\bu_t\cdot \bu+\eta\bu_t\cdot\bu-\frac{1}{6}\bu_t D^2\nabla\eta_t ~d\bx\\
&=\frac{d}{dt} \frac{1}{2}\int_\Omega g\eta^2+D|\bu|^2 + \eta|\bu|^2~d\bx\\
&=\frac{d}{dt} E(t)\ .
\end{align*}
It is noted that the key point for the conservation of energy is the particular choice 
of the parameter $\theta$ which ensures that $\tilde{a}+\tilde{b}=-(\tilde{c}+\tilde{d})$.

From  (\ref{eq:conservform}) we observe that $\nabla\times \bu_t=0$ since $\nabla\times\nabla Q=0$ for any smooth enough function $Q$. We conclude that the vorticity of the horizontal velocity is conserved in the sense $\nabla\times\bu=\nabla\times\bu_0$. Therefore, if the flow, initially, is irrotational, then it remains irrotational with $\nabla\times \bu=0$ for all $t\geq0$.

\subsection{Variational derivation}
The variational derivation of model equations appears to be attractive not only because of its simplicity but also because of the physical verification of the model and the energy conservation properties that can be obtained in trivial way. Here we follow the methodology introduced in \cite{SW1968,CD2012}. We first consider the following  approximations of the kinetic and potential energies: The shallow-water approximation of the kinetic energy is
$$\mathcal{K}=\frac{\rho}{2}\int_{t_1}^{t_2}\int_\Omega (D+\eta)|\bu|^2\ d\bx\, dt\ ,$$
and the analogous approximation of the potential energy is
$$\mathcal{V}=\frac{\rho}{2}\int_{t_1}^{t_2}\int_\Omega g\eta^2\ d\bx\, dt\ ,$$
where $\rho$ denotes the density of the water. 
We also consider the non-hydrostatic approximation of the conservation of mass 
$$\eta_t+\Div[(D+\eta)\bu] - \frac{1}{6} \Div (D^2 \nabla \eta_t)=0\ ,$$
where $H=D+\eta$ denotes the total depth of the water. 
Then, we define the action integral
$$\mathcal{I}=\mathcal{K}-\mathcal{V}+\rho\int_{t_1}^{t_2}\int_{\Omega}[\eta_t+\Div[(D+\eta)\bu] - \frac{1}{6} \Div (D^2 \nabla \eta_t)]\phi~d\bx\, dt\ ,$$
where we impose the mass conservation by introducing the Lagrange multiplier $\phi(x,t)$, which as we shall see in the sequel coincides with a velocity potential of the horizontal velocity $\bu$.

The Euler-Lagrange equations for the action integral $\mathcal{I}$ are then the following
\begin{align}
\delta \phi~: & \quad \eta_t+\Div[(D+\eta)\bu] - \frac{1}{6} \Div (D^2 \nabla \eta_t)=0\ , \label{eq:var1}\\
\delta \bu~: & \quad \bu-\nabla\phi=0\ , \label{eq:var2}\\
\delta \eta~: & \quad \frac{1}{2}|\bu|^2-g\eta-\phi_t+\frac{1}{6}\nabla\cdot(D^2\nabla\phi_t)-\bu\cdot\nabla\phi=0\ . \label{eq:var3}
\end{align}
Taking the gradient of all terms in (\ref{eq:var3}) and eliminating $\nabla\phi$ using (\ref{eq:var2}) 
we obtain the approximate momentum conservation equation
\begin{equation}\label{eq:var4}
 \bu_t+g\nabla \eta+ \frac{1}{2}\nabla|\bu|^2-\frac{1}{6}\nabla(\Div(D^2\bu_t))=0\ .
\end{equation}
The new BBM-BBM system consists of the approximations of mass conservation (\ref{eq:var1}) and momentum conservation (\ref{eq:var4}), and its solutions preserve the approximation of the total energy $\mathcal{E}=\mathcal{K}+\mathcal{V}$. We will call the new system regularized shallow water equations so as to differentiate from the other BBM-BBM systems.

\section{Well-posedness}\label{sec:wellposed}
\subsection{The flat bottom case}

In this section we study the well-posedness of the initial-boundary value problem (\ref{eq:bbm2a})--(\ref{eq:bcs})).  
For simplicity we first consider flat bottom topography $D(\bx)=D_0$ and with the same initial and boundary conditions as before.
For theoretical purposes we consider the system in dimensionless and scaled variables with $\varepsilon=\sigma^2$, and also we assume that the domain $\Omega\subset\R^2$ is smooth 
(at least piecewise smooth with no reentrant corners). The equations are simplified by dropping the tilde from the notation, and the initial-boundary value problem \eqref{eq:bbm2a}, \eqref{eq:ics}, \eqref{eq:bcs} can be written as 
\begin{equation}\label{eq:sys1}
\begin{aligned}
&\eta_t+\Div \bu +\varepsilon\Div(\eta \bu) -\varepsilon\frac{1}{6} \Delta \eta_t = 0 \ ,  \\
& \bu_t+\nabla \eta+\varepsilon\frac{1}{2}\nabla |\bu|^2-\varepsilon\frac{1}{6} \nabla(\Div \bu_t)= 0\ , \\
&\eta(\bx,0)=\eta_0(\bx), \quad \bu(\bx,0)=\bu_0(\bx), \quad \text{ on } \Omega\ , \\
&\nabla\eta\cdot \bn=0,  \quad \bu\cdot \bn=0, \quad \text{ on } \partial\Omega\ .
\end{aligned}
\end{equation}

For the purposes of this paper we will use the usual Sobolev space $H^1=H^1(\Omega)$ consisting of weakly differentiable functions on $\Omega$, 
and the space 
$$\bH^1_0(\Omega)=\{\bu\in H^1\times H^1:  \bu\cdot \bn=0 \ \mathrm{ on } \ \partial \Omega \}\ .$$
We equip the space $H^1$ with the usual $H^1$-norm defined for all $w\in H^1$ 
to be $\|w\|_1=(\|w\|^2+\|\nabla w\|^2)^{1/2}$, and the space $\bH^1_0$ 
with the norm $\|\bw\|_1=(\|w_1\|_1^2+\|w_2\|_1^2)^{1/2}$ for all $\bw=(w_1,w_2)^T\in \bH^1_0$. 
We will also denote the usual inner product of $L^2$ by $(\cdot,\cdot)$, and we will use the spaces $L^p$ and $\bL^p=L^p\times L^p$ for any $p>0$. Note that because we will always consider functions defined in a bounded domain $\Omega$ we will refrain from mentioning the domain in the notation of the functions spaces.

We will find also useful the Sobolev space $W_k^\infty(\Omega)=\{u\in L^1_{\text{loc}}(\Omega)\, :\, \|u\|_{k,\infty}<\infty\}$ where $\|u\|_{k,\infty}=\max_{|\alpha|\leq k} \|\partial^\alpha u\|_{L^\infty}$.

\begin{remark}\label{rem:remarkcurl}
Denoting $\|\bu\|_\div=(\|\bu\|^2+\|\Div \bu\|^2)^{1/2}$, we define the spaces 
$$H^{\div}(\Omega)=\{ \bu\in \bL^2(\Omega),\Div \bu\in L^2(\Omega) \}, 
\qquad H_0^{\div}(\Omega)=\{ \bu\in H^{\div}(\Omega): \bu\cdot \bn =0  \ \mathrm{ on } \ \partial \Omega \}\ ,$$ 
and 
$$H^{\curl}(\Omega)=\{ \bu\in \bL^2(\Omega), \Curl \bu\in L^2(\Omega) \}, 
\qquad H_0^{\curl}(\Omega)=\{ \bu\in H^{\curl}(\Omega): \bu\times \bn =0  \ \mathrm{ on } \ \partial \Omega \}.$$ 
It is known that for a domain $\Omega$ with appropriately smooth boundary, we have 
$$\|\bu\|_1\lesssim (\|\bu\|_{\div}^2+\|\Curl \bu\|^2)^{1/2},
\quad \mbox{ for } \bu\in H^{\div}_0(\Omega)\cap H^{\curl}(\Omega)\ . $$
For details on the properties of these particular spaces we refer to \cite{GR1986}.
\end{remark}

\begin{remark}
We will also consider the spaces
$$H^{\div}_s(\Omega)=\{\bu\in H^{\div}(\Omega), \Div \bu\in H^s(\Omega)\}, 
\quad H_{s,0}^\div(\Omega)=\{\bu \in H^\div_s(\Omega)\cap H^{\div}_0(\Omega)\}\ ,$$
equipped with the norm
\begin{equation}
\|\bu\|_{s,\div}=\left(\|\bu\|^2+\|\Div\bu\|_s^2\right)^{1/2}, \quad \mbox{ for }~ \bu\in H^\div_s\ .
\end{equation} 
These spaces are practically the departure spaces of the operator $I-\nabla( \nabla \cdot )$. 
We reserve the notation $H^2$ to denote the classical Sobolev space $W^{2,2}$.
Furthermore, we define the negative norms
$$\|\bu\|_{-s,\div}=\sup_{\bz \in H_{s,0}^{\div}, \bz\ne 0}\frac{(\bu,\bz)}{{\|\bz\|_{s,\div}}}\ ,
$$
while $\|u\|_{-s}$ denotes the standard dual norm in the Sobolev space $H^{s}$.
\end{remark}

We define the bilinear forms $a:H^1\times H^1\rightarrow \mathbb{R}$ 
and $b:H^\div_0\times H^\div_0\rightarrow \mathbb{R}$ as
\begin{align}
&a(u,v)=(u,v)+\varepsilon\frac{1}{6}(\nabla u,\nabla v), \quad \mbox{ for all } u, v\in H^1\ , \label{eq:wp1} \\
&b(\bu,\bv)=(\bu,\bv)+\varepsilon\frac{1}{6}(\Div \bu, \nabla \cdot \bv), \quad \mbox{ for all } \bu, \bv\in H^\div_0\ . \label{eq:wp2}
\end{align}

Then the weak formulation of the problem (\ref{eq:bbm2a}) is defined as follows: 
Seek $(\eta,\bu)\in H^1\times \bH^1_0$ such that
\begin{equation}\label{eq:varpsys}
\begin{aligned}
& a(\eta_t,\chi) + (\Div \bu,\chi)+\varepsilon(\Div (\eta\bu),\chi)=0, \quad \mbox{ for all } \chi\in H^1\ ,\\
& b(\bu_t,\bchi) + (\nabla \eta,\bchi) + \varepsilon\frac{1}{2}(\nabla |\bu|^2,\bchi)=0, \quad \mbox{ for all } \bchi\in \bH^1_0\ .
\end{aligned}
\end{equation}

A solution of (\ref{eq:varpsys}) is called a weak solution. Using the divergence 
theorem, it can be seen that any classical solution of the system (\ref{eq:sys1}) 
satisfies the weak formulation (\ref{eq:varpsys}), and thus classical solutions are 
also weak solutions. 

Before stating the main result of this paragraph, we define the mappings  $f:\bL^2 \rightarrow H^1$ and $g:L^2\rightarrow H^\div_0$ as follows
\begin{equation}\label{eq:definf}
a(f(\bw),\chi)=(\bw,\nabla \chi), \mbox{ for all } \bw\in \bL^2 \mbox{ and } \chi \in H^1\ ,
\end{equation}
and
\begin{equation}\label{eq:defing}
b(g(w),\bchi)= (w, \Div \bchi), \mbox{ for all } w\in L^2 \mbox{ and } \bchi \in H^\div_0\ .
\end{equation}

The mappings $f$ and $g$ are well defined. Indeed, it is not hard to see that they are continuous 
in $\bL^2$ and $L^2$, respectively, in the sense that $\|f(\bw)\|\lesssim \|\bw\|$ and $\|g(w)\|\lesssim \|w\|$, 
where $\lesssim$ denotes the inequality $\| \cdot\| \le C \| \cdot\|$ for an unspecified
positive constant $C$, independent of $\varepsilon$.
Specifically, we have the following lemma:
\begin{lemma}\label{lem:maps}
The operators $f$ and $g$  in \eqref{eq:definf} and \eqref{eq:defing} respectively, are well defined. 
Moreover, the following inequalities hold:
\begin{equation}\label{eq:ineq1a}
\|f(\bw)\|_1\lesssim \|\bw\|,  \quad \mbox{ for all } \bw\in \bL^2\ ,
\end{equation}
and
\begin{equation}\label{eq:ineq1b}
\|g(w)\|_{\rm div}\lesssim \|w\|, \quad \mbox{ for all } w\in L^2\ .
\end{equation}
Furthermore, $g(w)\in \bH^1_0$ and $\|g(w)\|_1\lesssim \|w\|$ for all $w\in L^2$.
\end{lemma}
\begin{proof}
The existence of the $f$ and $g$ (and also the inequalities (\ref{eq:ineq1a}) and (\ref{eq:ineq1b})) is a direct consequence of Riesz representation theorem. Moreover, the continuity of $f$  can be proven easily using the Cauchy-Schwarz inequality
$$
\|f(\bw)\|_1^2 \lesssim a(f(\bw),f(\bw)) = (\bw, \nabla f(\bw))\leq \|\bw\|\|f(\bw)\|_1\ ,
$$
and thus $\|f(\bw)\|_1\leq \|\bw\|$. Similarly, one can prove the inequality $\|g(w)\|_\div\lesssim\|w\|$ as well.
In addition, since (\ref{eq:defing}) holds for all $\bchi\in H^\div_0$, by choosing $\bchi\in \mathcal{D}(\bar{\Omega})^2$, (where $\mathcal{D}(\bar{\Omega})$ is the space of infinitely differentiable functions with compact support on $\Omega$), yields that $\Div g(w)-w\in H^1$ and
$$g(w)=\nabla(\Div g(w)-w)~ \mbox{ in } \bL^2\ ,$$ hence $\Curl g(w)=0$ in $\Omega$, (see also \cite{GR1986}, Thm. 2.9). Therefore, $g(w)\in \bH^1_0$, and due to Remark \ref{rem:remarkcurl} we conclude $\|g(w)\|_1\lesssim \|w\|$. 
\end{proof}

\begin{remark}\label{rem:maps}
Alternatively, we can reach to the same conclusion by observing that $g(w)$ 
is the solution $g(w)=(I-\varepsilon\frac{1}{6}\nabla \nabla \cdot ~)^{-1}\nabla w$ so that we have $\Curl g(w) =0$. 
\end{remark}

\begin{remark}\label{rem:maps2}
By the standard theory of elliptic equations \cite{triebel}, if $\Div \bw\in \bH^{s-1}$ for $s>\frac{3}{2}$ and $\bw\cdot\bn=0$ on $\partial\Omega$, then $f(\bw)\in H^{s+1}$ is the weak solution of the Neumann problem of the equation $(I-\varepsilon\frac{1}{6}\Delta)f(\bw)=-\Div\bw$ in $L^2$ with $\nabla f(\bw)\cdot\bn =0$ in $L^2(\partial\Omega)$. Thus, $f(\bw)=-(I-\varepsilon\frac{1}{6}\Delta \cdot ~)^{-1}\Div \bw$, where the operator $I-\varepsilon\frac{1}{6}\Delta$ has domain the space $X=\{v\in H^2: \nabla v\cdot \bn=0~\text{ on $\partial\Omega$}\}$. 
\end{remark}

\noindent
Now we are ready to prove the main result of this section.

\begin{theorem}\label{thrm:main}
For any initial conditions $(\eta_0,\bu_0)\in H^1\times \bH^1_0$, there exists a maximal time $T>0$, independent of $\varepsilon$, and a unique weak solution
$(\eta,\bu)\in C^1([0,T]; H^1)\times C^1([0,T]; \bH^1_0)$ of the initial-boundary value problem (\ref{eq:sys1}).
\end{theorem}
\begin{proof}
With the help of the mappings $f$ and $g$ we write (\ref{eq:wp1}) and (\ref{eq:wp2}) 
as a system of ordinary differential equations in the distributional sense
\begin{align}
& \eta_t= f(\bu)+\varepsilon f(\eta\bu)\ , \label{eq:massw}\\
& \bu_t=g(\eta)+\varepsilon\frac{1}{2}g(|\bu|^2)\ , \label{eq:momw}
\end{align}
or in the more compact form
\begin{equation}\label{eq:odesystem}
U_t=F(U)\ ,
\end{equation}
where $U=(\eta,\bu)^T$ and 
\begin{equation}
F(U)=\left(f(\bu)+\varepsilon f(\eta\bu),  g(\eta)+\varepsilon\frac{1}{2}g(|\bu|^2)\right)^T\ .
\end{equation}
If $\eta\in H^1$ and $\bu\in \bH^1_0$ then $\eta\bu\in \bL^2$ and $|\bu|^2\in L^2$ due to Grisvard's lemma \cite{Grisvard} (see alternatively \cite{BH2015}) and thus the function $F$ is well-defined. Moreover, 
since $f$ maps its argument into $H^1$ and $g$ into $\bH^1_0$ we deduce that $F$ is $C^1$ on $H^1\times \bH^1_0$, with derivative $F'(\eta^{\ast},\bu^{\ast})$ given by
\begin{equation}
F'(\eta^{\ast},\bu^{\ast})(\eta,\bu) = \begin{pmatrix}
f(\bu)+\varepsilon f(\eta\bu^\ast)+\varepsilon f(\eta^{\ast}\bu)\\
g(\eta)+\varepsilon g(\bu^\ast\cdot \bu)
\end{pmatrix}\ .
\end{equation}
The continuity of $F'$ follows from the continuity of $f$ and $g$: Let $U=(\eta,\bu)^T\in H^1\times \bH^1_0$, then using Lemma \ref{lem:maps} we have,
\begin{align*}
\|F'(\eta^\ast,\bu^\ast)U\|_1 &=\sqrt{\|f(\bu)+\varepsilon f(\eta\bu^\ast)+\varepsilon f(\eta^{\ast}\bu)\|_1^2+\|g(\eta)+\varepsilon g(\bu^\ast\cdot \bu))\|_1^2}\\
&\leq  \sqrt{\|\bu\|^2+\|\eta\bu^\ast\|^2+\|\eta^\ast\bu\|^2+\|\eta\|^2+\|\bu^\ast\cdot\bu\|^2}\\
&\leq  \sqrt{\|\bu\|^2+\|\eta\|^2+\|\eta\|_{L^4}^2\|\bu^\ast\|_{\bL^4}^2+\|\eta^\ast\|_{L^4}^2\|\bu\|_{\bL^4}^2+\|\bu^\ast\|_{\bL^4}^2\|\bu\|_{\bL^4}^2}\\
&\lesssim \|U\|_1 \ ,
\end{align*}
where we have used the following Gagliardo-Nirenberg inequality \cite{brezis2010},
$$\|w\|_{L^4}\lesssim \|w\|^{1/2}\|w\|_1^{1/2}\lesssim \|w\|_1, \quad w\in H^1\ .$$
Taking $(\eta^\ast,\bu^\ast)\in H^1\times \bH_0^1$ we deduce that $F'(\eta^\ast,\bu^\ast)$ is continuous.
Thus, from the theory of ordinary differential equations in Banach spaces (cf. e.g. \cite{Berger77,brezis2010}), we have that for any initial conditions $(\eta_0,\bu_0)\in H^1\times \bH^1_0$, there exists a maximal time $T=T(\varepsilon)>0$ and a unique solution
$(\eta,\bu)\in C^1([0,T]; H^1)\times C^1([0,T]; \bH^1_0)$ of the initial-boundary value problem (\ref{eq:sys1}).

To prove that the maximal time $T$ is independent of $\varepsilon$, first we observe that the solution $(\eta,\bu)$ of the initial-boundary value problem (\ref{eq:bbm2a})--(\ref{eq:bcs}) satisfies the following energy conservation:
\begin{equation}\label{eq:energy}
\frac{1}{2}\frac{d}{dt}\int_\Omega \left[\eta^2+|\bu|^2+\frac{\varepsilon}{6} \left(|\nabla\eta|^2+[\Div \bu]^2\right)\right] =\varepsilon \int_\Omega \eta\bu\cdot \nabla\eta+\frac{\varepsilon}{2}\int_\Omega |\bu|^2\Div \bu\ .
\end{equation}
Defining
$$I_\varepsilon(t)=(1-\frac{\varepsilon}{6})(\|\eta\|^2+\|\bu\|^2)+\frac{\varepsilon}{6}(\|\eta\|_1^2+\|\bu\|_\div^2)\ ,$$
we rewrite (\ref{eq:energy}) in the form
$$\frac{1}{2}\frac{d}{dt}I_\varepsilon=\varepsilon \int_\Omega \eta\bu\cdot \nabla\eta+\frac{\varepsilon}{2}\int_\Omega |\bu|^2\Div \bu\ .$$
Using H\"{older}'s inequality we have
\begin{equation}\label{eq:en1}
\left| \varepsilon \int_\Omega \eta\bu\cdot \nabla\eta+\frac{\varepsilon}{2}\int_\Omega |\bu|^2\Div \bu \right| \lesssim \varepsilon \|\nabla\eta\|\|\bu\|_{\bL^4}\|\eta\|_{L^4}+\varepsilon\|\bu\|_{\bL^4}^2\|\Div\bu\|\ .
\end{equation}
From (\ref{eq:en1}) and using the Gagliardo-Nirenberg, it follows
$$
\left|\frac{d}{dt}I_\varepsilon \right| \lesssim \varepsilon\|\eta\|_1^{3/4}\|\eta\|_1^{3/2}\varepsilon^{1/4}\|\eta\|^{1/4}\|\bu\|^{1/2}\|\bu\|_1^{1/2}+\varepsilon\|\bu\|\|\bu\|^2_1\ .
$$
Using Young's inequality we obtain
\begin{align*}
\left|\frac{d}{dt}I_\varepsilon \right|
&\lesssim \varepsilon^{3/2} \|\eta\|_1^3+\varepsilon^{1/2}\|\eta\|\|\bu\|\|\bu\|_1+\|\bu\|^3+\varepsilon^{3/2}\|\bu\|_1^3\\
&\lesssim \varepsilon^{3/2} \|\eta\|_1^3+\|\eta\|^{3/2}\|\bu\|^{3/2}+\varepsilon^{3/2}\|\bu\|_1+\|\bu\|^3+\varepsilon^{3/2}\|\bu\|_1^3\\
&\lesssim \|\eta\|^3+\|\bu\|^3+ \varepsilon^{3/2}(\|\eta\|_1^3+\|\bu\|_1^3)\\
(\text{since $\Curl \bu=\text{const}$}) &\lesssim \|\eta\|^3+\|\bu\|^3+ \varepsilon^{3/2}(\|\eta\|_1^3+\|\bu\|_\div^3) \ ,
\end{align*}
which implies
$$\frac{d}{dt}I_\varepsilon(t) \lesssim I_\varepsilon^{3/2}(t)\ .$$
The last inequality gives the {\em a priori} bound
\begin{equation}\label{eq:difineq}I_\varepsilon(t)\leq \frac{I_\varepsilon(0)}{\left(1-Ct \sqrt{I_\varepsilon(0)}\right)^2}\ .
\end{equation}
Since 
$$I_\varepsilon(0)=\int_\Omega \left[|\bu_0|^2+\eta_0^2+\frac{\varepsilon}{6} \left(|\nabla\eta_0|^2+[\Div \bu_0]^2\right)\right]\ ,$$
we have that $I_0(0)\leq I_\varepsilon(0)\leq I_1(0)$ for $0\leq \varepsilon \leq 1$ and thus
$$I_\varepsilon(t)\leq \frac{I_1(0)}{\left(1-Ct \sqrt{I_1(0)}\right)^2}\ .$$
on a time interval $[0,\tilde{T})$ where $\tilde{T}=O\left(1/\sqrt{I_1(0)}\right)$ independent of $\varepsilon$. Therefore, the maximal time of existence of the solution $(\eta,\bu)$ can be extended up to $\tilde{T}$.
Hence, we conclude that for $0<\varepsilon\ll1$, the maximal time $T$ is independent of $\varepsilon$.
\end{proof}

\begin{remark}
Note, that although the slip-wall boundary condition is satisfied by $\bu$, since $\bu\in \bH^1_0$, this is not obvious for the Neumann boundary condition of $\eta$. Since $\Div(\eta\bu)\in \bH^{s-1}$ for $s<1$ (see Grisvard's lemma) we have that $f(\bu+\varepsilon\eta\bu)$ is in $H^{s+1}$ for all $s<1$. Using (\ref{eq:massw}) we see that the trace of the normal derivative $\partial/\partial \bn$ on $\eta_t$ makes sense in $L^2(\partial\Omega)$, \cite{triebel}. By the Remark \ref{rem:maps2} we have that $\nabla f(\bu+\varepsilon\eta\bu)\cdot \bn=0$, and thus the solution $\eta$ satisfies the weak Neumann boundary condition $\nabla \eta_t\cdot \bn=0$ on $\partial \Omega$. Therefore, solutions of (\ref{eq:varpsys}) with the requisite  regularity (for example $C^2$), and with appropriate compatibility conditions satisfied by the initial conditions, automatically satisfy the boundary condition $\nabla\eta\cdot\bn=0$ on $\partial\Omega$ in a strong sense.
\end{remark}

\begin{remark}
Due to the regularity properties of the operator $I-\nabla(\nabla \cdot)$, 
we conclude that if the initial conditions are $(\eta_0,\bu_0)\in H^2\times (H^\div_{1,0}\cap \bH^1)$, then there exists a maximal time $T$ and a unique solution $(\eta,\bu)\in H^2\times (H^\div_{1,0}\cap \bH^1)$ of the initial-boundary value problem  (\ref{eq:sys1}) for $t\leq T$. Moreover, after multiplying the mass equations with $-\Delta\eta$ and the momentum equation with $\bu$, and using again the irrotationality of $\bu$ and the divergence theorem, we obtain the Bernoulli-type inequality
$$\frac{1}{2}\frac{d}{dt} Y_\varepsilon(t)\lesssim Y_\varepsilon(t)+ Y_\varepsilon^2(t)\ ,$$
for
$$Y_\varepsilon(t)=\|\bu\|^2+\frac{\varepsilon}{6}\|\Div \bu\|^2+\|\nabla\eta\|^2+\frac{\varepsilon}{6}\|\Delta\eta\|^2\ .$$ Solving this inequality we obtain an upper bound of the solution in $H^2\times \bH^1$. Similarly, we can obtain bounds of the solution in $H^2\times (H^\div_{1,0}\cap \bH^1)$. Note also that given sufficient smoothness, if the initial condition satisfies the compatibility condition $\nabla \eta_0\cdot\bn=0$ on $\partial\Omega$ then following Theorem \ref{thrm:main}, the solution will satisfy the Neumann condition $\nabla\eta\cdot\bn=0$.
\end{remark}

\begin{remark}
Local in time well-posedness of the Cauchy problem of similar Boussinesq systems to the one we studied here 
has been established in \cite{DMS2007} and in bounded domains with $\bu=0$ on $\partial\Omega$ 
in \cite{DMS2009,DMS2010}. In these cases one can show that the maximal time can be extended 
up to times of order $1/\sqrt{\varepsilon}$. In \cite{saut2012cauchy}, 
it was shown that the solution can be extended to times of $O(1/\varepsilon)$ 
if the domain is $\mathbb{R}^n$ and the initial conditions are of small amplitude. 
While these results also hold for the Cauchy problem associated to the system (\ref{eq:bbm2a}),
it is not obvious whether they can be extended to the case of bounded domains,
and we leave this question for future work.
\end{remark}

\subsection{The variable bottom case}

The previous analysis carries over to the case of general bottom topography under the assumption of mild bottom variations $D(\bx)=1+\beta D_b(\bx)\in W_1^\infty(\Omega)$ with $\beta\ll 1$.  As we shall see soon, the choice of the parameter $\beta$ can be very important.
 Consider the initial-boundary value problem \eqref{eq:bbm2a}, \eqref{eq:ics}, \eqref{eq:bcs} in nondimensional and scaled variables written as
\begin{equation}\label{eq:varbottopeps}
\begin{aligned}
& \begin{aligned} & \eta_t+\Div((D+\varepsilon\eta)\bu) - \frac{\varepsilon}{6} \Div (D^2 \nabla \eta_t)=0,  \\
& \bu_t+\nabla\eta+ \frac{\varepsilon}{2}\nabla|\bu|^2-\frac{\varepsilon}{6}\nabla(\Div(D^2\bu_t))=0,\end{aligned} \quad \text{ on }\Omega\ , \\
&\eta(\bx,0)=\eta_0(\bx), \quad \bu(\bx,0)=\bu_0(\bx), \quad \text{ on } \Omega\ , \\
&\nabla\eta\cdot \bn=0,  \quad \bu\cdot \bn=0, \quad \text{ on } \partial\Omega\ .
\end{aligned}
\end{equation}
In this case, we multiply the momentum equation with $D^2$. 
The weak formulation of the initial-boundary value problem 
(\ref{eq:bbm2a})--(\ref{eq:bcs}) then becomes: \\
Seek  $(\eta,\bu)\in (H^1,\bH^1_0)$ such that
\begin{equation}\label{eq:gencase}
\begin{aligned}
& a_D(\eta_t,\chi)+(\Div ((D+\varepsilon\eta)\bu,\chi)=0, \quad \forall \chi\in H^1\ ,\\
& b_D(\bu_t,\bchi)+(D^2\nabla \eta,\bchi)+\frac{\varepsilon}{2}(D^2\nabla |\bu|^2,\bchi)=0\quad \forall \bchi\in\bH^1_0\ ,
\end{aligned}
\end{equation}
where
\begin{equation}
\begin{aligned}
& a_D(u,v)=(u,v)+\frac{\varepsilon}{6}(D\nabla u,D\nabla v),\quad \forall u,v\in H^1\ , \\
& b_D(\bu,\bv)=(D\bu,D\bv)+\frac{\varepsilon}{6}(\Div (D^2\bu),\Div (D^2 \bv)),\quad \forall \bu,\bv\in \bH^1_0\ .
\end{aligned}
\end{equation}
We show bellow that the standard ``non-cavitation assumption'' is enough to guarantee well-posedness.
\begin{lemma}\label{lem:newbilinears}
Let $\varepsilon,\beta>0$ positive and small. If the bottom topography $D=1+\beta D_b\in W_1^\infty$ and also we assume for simplicity that
$$0<D_m\leq D(\bx)\leq D_M ,\quad\text{ and }\quad 0<\beta D'_m \leq |\nabla D(\bx)|\leq \beta D_M'\ , $$ 
then the bilinear forms $a_D$ and $b_D$ are continuous and coercive.
\end{lemma}
\begin{proof}
The continuity of $a_D$ and $b_D$ is straightforward under the assumption of bounded bottom topography $\|D\|_{1,\infty}< \infty$. 
then $a_D$ and $b_D$ are coercive as well. The coerciveness of $a_D$ is trivial while the coerciveness of $b_D$ can be shown as follows. For $\bu\in \bH^1_0$ we have that 
$$b_D(\bu,\bu)=\|D\bu\|^2+\frac{\varepsilon}{6}\|\Div (D^2\bu)\|^2\ .$$
Then, we have
$$
\begin{aligned}
\|\bu\|^2+\frac{\varepsilon}{6}\|\Div\bu\|^2 &= \|D^{-1}D\bu\|^2+\frac{\varepsilon}{6}\|\nabla\cdot(D^{-2}D^2\bu)\|^2\\
&\leq D_m^{-2}\|D\bu\|^2+\frac{\varepsilon}{6}\|\nabla(D^{-2})\cdot D^2\bu+D^{-2}\Div (D^2\bu)\|^2 \\
&\leq D_m^{-2}\|D\bu\|^2+\frac{\varepsilon}{3}\left(\|(D^{-1}\nabla D)^2\|_\infty\|D\bu\|^2+D_m^{-4}\|\Div (D^2\bu)\|^2 \|\right)\\
&\leq C\left(D_m^{-1},\varepsilon(\beta D_M')^2\right) b_D(\bu,\bu)\ .
\end{aligned}
$$
This implies that 
$$b_D(\bu,\bu)\geq C\|\bu\|_\div^2\ ,$$
where 
$C>0$ depends on $\varepsilon$ and $\beta$, and also on the bounds of the depth function and its gradient.
\end{proof}

Similarly to the flat bottom case, we generalize the mappings  $f:\bL^2 \rightarrow H^1$ and $g:L^2\rightarrow H^\div_0$ to include the general bottom topography 
\begin{equation}\label{eq:definfd}
a_D(f(\bw),\chi)=(\bw,\nabla \chi), \mbox{ for all } \bw\in \bL^2 \mbox{ and } \chi \in H^1\ ,
\end{equation}
and
\begin{equation}\label{eq:defingd}
b_D(g(w),\bchi)= (w, \Div (D^2\bchi)), \mbox{ for all } w\in L^2 \mbox{ and } \bchi \in H^\div_0\ .
\end{equation}

The mappings $f$ and $g$ satisfy Lemma \ref{lem:maps} with an additional hypothesis on the bottom topography $D$. Specifically, $f$ and $g$ satisfy the following lemma:
\begin{lemma}\label{lem:maps2}
If $D\in W_1^\infty$ then the operators $f$ and $g$  in \eqref{eq:definfd} and \eqref{eq:defingd} respectively, are well defined. 
Moreover, the following inequalities hold:
\begin{equation}\label{eq:ineq1ad}
\|f(\bw)\|_1\lesssim \|\bw\|,  \quad \mbox{ for all } \bw\in \bL^2\ ,
\end{equation}
and
\begin{equation}\label{eq:ineq1bd}
\|g(w)\|_{\rm div}\lesssim \|w\|, \quad \mbox{ for all } w\in L^2\ .
\end{equation}
Furthermore, $g(w)\in \bH^1_0$ and $\|g(w)\|_1\lesssim \|w\|$ for all $w\in L^2$.
\end{lemma}
\begin{proof} 
The proof follows immediately from the properties of $a_D$, $b_D$ and the definitions (\ref{eq:definfd}), (\ref{eq:defingd}).
\end{proof}

Theorem \ref{thrm:main} can be also extended in the general case of variable bottom topography with mild variations to the following theorem:
\begin{theorem}\label{thrm:main2}
If the bottom topography $D$ is as in Lemma \ref{lem:newbilinears}, then for any initial conditions $(\eta_0,\bu_0)\in H^1\times \bH^1_0$ that satisfy $\nabla\eta_0\cdot\bn=0$, there exists a maximal time $T>0$, independent of $\varepsilon$, and a unique weak solution
$(\eta,\bu)\in C^1([0,T]; H^1)\times C^1([0,T]; \bH^1_0)$ of the initial-boundary value problem, \eqref{eq:varbottopeps}.
\end{theorem}
\begin{proof}
First observe that system (\ref{eq:gencase}) can be written as
\begin{align}
& \eta_t= f(D\bu)+\varepsilon f(\eta\bu)\ ,\\
& \bu_t=g(\eta)+\varepsilon\frac{1}{2}g(|\bu|^2)\ ,
\end{align}
which has the same exact form as system (\ref{eq:odesystem}) in Theorem \ref{thrm:main}. The rest of the proof is very similar to the proof of Theorem \ref{thrm:main} and is omitted. 
\end{proof}

In the next section we explore the properties of the system (\ref{eq:bbm2a})--(\ref{eq:bcs}) 
using the standard Galerkin finite element method.

\section{Finite element discretization and error estimates}\label{sec:numsols}

Let $\Omega$ be a convex polygonal domain. We consider a regular triangulation $\mathcal{T}_h=\{\tau_j\}_{j=1}^{N_h}$ of $\Omega$ comprised of $N_h$ triangles $\tau$ such that $\Omega=\cup_{j=1}^{N_h}\tau_j$. We denote the maximum side length of the triangulation by $h, \ 0<  h <1$. 
We consider the standard finite-dimensional space $S_h^r$, for any integer $r\geq 2$, with $S_h^r \subset C(\bar{\Omega})\cap H^1$,
having the following approximation property: For any $w\in H^s$ the identity 
\begin{equation}\label{eq:femident}
\inf_{\chi\in S_h^r} \left\{\|w-\chi\|+h\|w-\chi\|_1 \right\}\lesssim h^s\|w\|_s,\qquad 1\leq s\leq r\ ,
\end{equation}
holds for small enough $h$. We consider the space $\bS_h^p=S_h^p\times S_h^p$, and we define the semi-discretization of system (\ref{eq:bbm2a})--(\ref{eq:bcs}) as the problem of finding $(\eta^h,\bu^h)\in S_h^r\times \bS_h^p$ that satisfy for all $h>0$
\begin{equation}\label{eq:semidisc}
\begin{aligned}
&\mathcal{A}(\eta^h_t,\chi)-((D+\eta^h)\bu^h,\nabla\chi)=0,\quad \mbox{ for all } \chi\in S_h^r\ , \\
&\mathcal{B}(\bu^h_t,\boldsymbol{\chi})+( \nabla(g\eta^h+\tfrac{1}{2}|\bu^h|^2),D^2\boldsymbol{\chi})=0,\quad \mbox{ for all } \boldsymbol{\chi}\in \bS_h^p\ ,
\end{aligned}
\end{equation}
for appropriate values of $r$ and $p$ 
and with the symmetric bilinear forms $\mathcal{A}$ and $\mathcal{B}$ defined as
\begin{align}
\mathcal{A}(\phi,\chi)&=(\phi,\chi)+\frac{1}{6}(D\nabla\phi,D\nabla\chi),\quad \mbox{ for } \phi,\chi\in S_h^r\ ,\\
\mathcal{B}(\boldsymbol{\phi},\boldsymbol{\chi})&=(D\boldsymbol{\phi},D\boldsymbol{\chi})+\frac{1}{6}(\Div (D^2\boldsymbol{\phi}),\Div (D^2\boldsymbol{\chi})) -\frac{1}{6}\langle \Div (D^2 \bphi) ,D^2 \bchi\cdot\bn\rangle \label{eq:bformu} \\
&\quad -\frac{1}{6} \langle D^2\bphi\cdot\bn,\Div(D^2 \bchi)\rangle +\frac{\gamma}{h} \langle D^2\boldsymbol{\phi}\cdot\bn,\boldsymbol{\chi}\cdot \bn\rangle,\quad \mbox{ for } \boldsymbol{\phi},\boldsymbol{\chi}\in \bS_h^p\ , \nonumber
\end{align}
where $\gamma/h\gg 1$, and 
$$\langle f,g\rangle =\int_{\partial\Omega} f g~ ds\ ,$$
is the usual $L^2$ inner product on the boundary $\partial\Omega$.
The system (\ref{eq:semidisc}) is also accompanied by smooth initial data $(\eta^h(\bx,0),\bu^h(\bx,0))=(\eta^h_0(\bx),\bu^h_0(\bx))$. The function $(\eta^h_0(\bx),\bu^h_0(\bx))$ can be taken as a projection or interpolant of the actual initial data $(\eta_0(\bx),\bu_0(\bx))$ onto $S_h^r\times \bS_h^p$. Note that we consider the problem in dimensional variables because apart from the fact that some parameters will depend on $\varepsilon$ and $\beta$ these parameters do not play any significant role in the numerical analysis of the problem. 

\begin{remark}
In addition to the inner product $\langle f,g\rangle$ we consider the norm $\|f\|_{\partial\Omega}=\sqrt{\langle f,f\rangle}$ whenever the trace of $f$ on $\partial\Omega$ makes sense, for example if $f\in H^1(\Omega)$, \cite{Brenner07}. 
\end{remark}
\begin{remark}
The first boundary integral term in (\ref{eq:bformu}) occurs because the space $\bS_h^p$ does not satisfy the slip-wall boundary condition of the continuous problem. On the other hand the next term is expected to be zero because the solution should satisfy the slip-wall boundary condition and makes the bilinear form symmetric. The third boundary integral term is the heart of Nitsche's method that forces the solution to satisfy the slip-wall boundary condition. The parameters in front of the boundary integral terms in (\ref{eq:bformu}) have been chosen equal so as to make the bilinear form symmetric. This does not affect the proofs in the sequel, though it is advantageous in terms of the matrix properties and linear systems solvers.
\end{remark}

\subsection{A Galerkin method for the incomplete-elliptic problem}\label{ellpiticproj}

Throughout this section we assume that the bottom satisfies the conditions of Theorem \ref{thrm:main2}. The specific weak formulation of the original problem is an adaptation of Nitsche's method. 
In order to analyze the specific finite element discretization 
we closely follow the ideas of \cite{Thomee}. We define the norm
$$\norma{\bu}=\left(\|\bu\|_\div^2+ h\|\Div \bu\|_{\partial\Omega}^2+h^{-1}\|\bu\cdot\bn\|_{\partial\Omega}^2\right)^{1/2}\ ,$$
for any $\bu\in H^\div$.  This norm is equivalent to $\|\Div \bu\|$ in $\bS_h^p$ since  (see \cite{monk})
\begin{equation}\label{eq:divinveq}
\|\bchi\cdot\bn\|_{\partial \Omega}\lesssim \|\Div\bchi\| \qquad \text{ and } 
\qquad \|\Div\bchi\|_{\partial\Omega}\leq C_0 h^{-1/2}\|\Div\bchi\| , \quad \forall \bchi \in \bS_h^p\ .
\end{equation}
Note that the hidden constants in the symbol $\lesssim$ are independent of $h$.
It is then straightforward to see that the symmetric bilinear form $\mathcal{B}$ is continuous and coercive. 
\begin{lemma}\label{lem:bilinear}
For sufficiently large value of $\gamma\gg 1$ and for any $\bphi,\bpsi\in\bS_h^p$ 
it can be shown that
\begin{equation}
|\mathcal{B}(\bphi,\bpsi)|\lesssim \norma{\bphi}\norma{\bpsi}\ , \qquad\text{continuity}, 
\end{equation}
and also
\begin{equation}
\mathcal{B}(\bphi,\bphi)\gtrsim \norma{\bphi}^2\ , \qquad\text{coercivity}.
\end{equation}
\end{lemma}
\begin{proof}
By the definition of $\mathcal{B}$ we have
\begin{align*}
|\mathcal{B}(\bphi,\bpsi)| &\lesssim \|\bphi\|\|\bpsi\|+\|\Div\bphi\|\|\Div\bpsi\|
+ \|\bphi\|\|\Div\bpsi\|+\|\Div\bphi\|\|\bpsi\|+\\
&\quad + \|\bphi\|_{\partial\Omega}\|\bpsi\cdot\bn\|_{\partial\Omega} +\|\Div \bphi\|_{\partial\Omega} \|\bpsi\cdot \bn\|_{\partial\Omega}+ \|\bphi\cdot \bn\|_{\partial\Omega}\|\Div \bpsi\|_{\partial\Omega}+\|\bphi\cdot\bn\|_{\partial\Omega}\|\bpsi\|_{\partial\Omega} \\
&\quad +h^{-1}\|\bphi\cdot\bn\|_{\partial\Omega}\|\bpsi\cdot\bn\|_{\partial\Omega}\\
 &\lesssim \|\bphi\|^2+\|\bpsi\|^2+ \|\Div\bphi\|^2+\|\Div\bpsi\|^2+ \\
 &\quad+   h^{1/2}\| \bphi\|_{\partial\Omega} h^{-1/2} \|\bpsi\cdot \bn\|_{\partial\Omega}+ 
 h^{1/2}\|\Div \bphi\|_{\partial\Omega} h^{-1/2} \|\bpsi\cdot \bn\|_{\partial\Omega}+\\
 &\quad+  h^{-1/2}\|\bphi\cdot \bn\|_{\partial\Omega} h^{1/2}\|\Div \bpsi\|_{\partial\Omega} 
+h^{-1/2}\|\bphi\cdot\bn\|_{\partial\Omega}h^{-1/2}\|\bpsi\cdot\bn\|_{\partial\Omega}\\
&\lesssim \norma{\bphi}\norma{\bpsi}\ .
\end{align*}
The second inequality follows similarly from the definition of $\mathcal{B}$ and the norm $\norma{\cdot}$:
$$
\begin{aligned}
\mathcal{B}(\bphi,\bphi) &= \|D\bphi\|^2+\frac{1}{6}\|\Div (D^2\bphi)\|^2-\frac{1}{3} \langle D^2\bphi\cdot\bn,\Div(D^2 \bphi)\rangle+\frac{\gamma}{h}\|D\bphi\cdot \bn\|_{\partial\Omega}^2\\
&\geq C_1\|\bphi\|^2+C_2\|\Div\bphi\|^2-C_3\|\bphi\cdot\bn\|_{\partial\Omega}\|\bphi\|_{\partial\Omega}-C_4\|\bphi\cdot\bn\|_{\partial\Omega}\|\Div\bphi\|_{\partial\Omega}+ \gamma C_5 h^{-1}\|\bphi\cdot\bn\|^2_{\partial\Omega}\\
&\geq C_1\|\bphi\|^2+C_2\|\Div \bphi\|^2 -C_3\|\bphi\cdot\bn\|_{\partial\Omega}^2 - \frac{C_2}{2 C_0^2}h \|\Div\bphi\|^2_{\partial\Omega} - \frac{C_0^2 C_4^2}{2C_2} h^{-1}    \|\bphi\cdot\bn\|_{\partial\Omega}^2+\gamma C_5 h^{-1}\|\bphi\cdot\bn\|_{\partial\Omega}^2\\
&\geq C_1\|\bphi\|^2+C_2\|\Div \bphi\|^2-\frac{C_2}{2 C_0^2} h\|\Div\bphi\|^2_{\partial\Omega}+\left[\gamma C_5 - h C_3-\frac{C_0^2C_4^2}{2C_2}\right]h^{-1}\|\bphi\cdot\bn\|^2_{\partial\Omega},
\end{aligned}
$$
where the constants $C_i = C_i(D), \ i=0,1,\dots,7$ with $C_0$ the constant in (\ref{eq:divinveq}), $C_1$, $C_2$ as in Lemma \ref{lem:newbilinears}, $C_3=\frac{\beta}{3} D_M^3 D_M'$, $C_4=\frac{1}{3}D_M^4$ and $C_5=D_m^2$. By choosing appropriate $\gamma> 1$ and $h<1$, we can have  $C_6:=\gamma C_5 - h C_3-\frac{C_0^2C_4^2}{2C_2} > 0$.  Also, by denoting $C_7 = \frac{C_2}{4C_0^2}$, and using the second (inverse) inequality of (\ref{eq:divinveq}) we obtain
\begin{equation*}
\begin{aligned}
\mathcal{B}(\bphi,\bphi) &\geq C_1\|\bphi\|^2+\frac{C_2}{4}\|\Div \bphi\|^2+C_7 h\|\Div\bphi\|^2_{\partial\Omega}+C_6 h^{-1}\|\bphi\cdot\bn\|^2_{\partial\Omega}\\
& \gtrsim\|\bphi\|^2+\|\Div \bphi\|^2+h\|\Div\bphi\|_{\partial\Omega}^2+ h^{-1}\|\bphi\cdot \bn\|_{\partial\Omega}^2 \\
&=\norma{\bphi}^2\ .
\end{aligned}
\end{equation*}
This completes the proof which shows that the bilinear form $\mathcal{B}$ is continuous and coercive.
\end{proof}

\begin{remark}
For sufficiently small $h< 1$ we have
\begin{equation}\label{eq:neqinvers}
\norma{\bchi}\lesssim h^{-1/2}\|\bchi\|_\div,  \  \forall \bchi\in\bS_h^p\ .
\end{equation}
\end{remark}
We will also need the following lemmata generalizing related results of \cite{Thomee} for vector valued functions:
\begin{lemma}\label{lem:compar}
If $\bw\in \bH^1\cap H^\div_{s,0}$ with $s\geq 1$ and $\bv=\bw-\bchi$ for $\bchi\in\bS_h^p$, then
$$\norma{\bv}\lesssim h^{-1}\left(\|\bv\|+h\|\bv\|_1+h^2|\bv|_{2,h} \right)\ ,$$
where $|\bv|_{2,h}$ denotes the norm
$$|\bv|_{2,h}=\left(\sum_{j=1}^{N_h}\|\nabla\Div\bv\|_{\bL^2(\tau_j)} \right)^{1/2}\ .$$
\end{lemma}
\begin{proof}
The proof follows from the facts that
$$h^{-1}\|\bv\cdot\bn\|^2_{\partial\Omega}\lesssim h^{-2}\|\bv\|^2+\|\bv\|^2_1\ ,$$
and
$$h\|\Div\bv\|^2_{\partial\tau}\lesssim \|\Div\bv\|^2_\tau+h^2\|\nabla\Div\bv\|^2_\tau\ ,$$
for $\tau\in\mathcal{T}_h$.
These inequalities can be proved using the trace inequality $\|v\|^2_{\partial\Omega}\lesssim \|v\|\|v\|_1$ of \cite{Brenner07} and the estimate $\|v\|_1\lesssim h^{-1}\|v\|+\|\nabla v\|$,  (see Lemma 2.3 of \cite{Thomee}).
\end{proof}

Now, we obtain the basic approximation property of the space $\bS_h^p$.

\begin{lemma}\label{lem:approxb}
The space $\bS_h^p$ equipped with the norm $\norma{\cdot}$ satisfies the following approximation property:
\begin{equation}
\inf_{\bchi\in\bS_h^p}\norma{\bw-\bchi}\lesssim h^{s-1}\|\bw\|_{s}, \quad \mbox{ for } ~ \bw\in \bH^{s}_0,\quad 2\leq s\leq p\ .
\end{equation}
\end{lemma}
\begin{proof}
It is known, \cite{Thomee}, that there is an interpolant $\bI_h$ into $\bS_h^p$ that satisfies
$$\|\bw-\bI_h\bw\|+h\|\bw-\bI_h\bw\|_1+h^2|\bw-\bI_h\bw|_{2,h}\lesssim h^{s}\|\bw\|_{s}, \quad \mbox{ for }~\bw\in \bH^{s},\quad 2\leq s\leq p\ .$$
We therefore then conclude that
\begin{align*}
\inf_{\bchi\in\bS_h^r}\norma{\bw-\bchi} &\lesssim \norma{\bw- \bI_h\bv} \\
&\lesssim h^{-1} \left(\|\bw-\bI_h\bw\|+h\|\bw-\bI_h\bw\|_1+h^2|\bw- \bI_h\bw|_{2,h}\right)\\
&\lesssim h^{s-1}\|\bw\|_{s}\ ,
\end{align*}
which completes the proof.
\end{proof}

Coming back to the semidiscrete problem, we  consider  only  initial conditions of the form
\begin{equation}
(\eta_0^h(\bx),\bu_0^h(\bx))=(R_h\eta_0(\bx), \bR_h\bu_0(\bx))\ ,
\end{equation}
where $R_h:H^1\to S_h^r$ and $\bR_h:H^\div\to \bS_h^p$ are the elliptic projections 
onto $S_h^r$ and $\bS_h^p$ respectively, defined as follows
\begin{align}
&\mathcal{A}(R_h w,\chi)=\mathcal{A}(w,\chi),\quad \forall  w\in H^1,~\chi\in S_h^r\ ,\\
&\mathcal{B}(\bR_h \bw,\boldsymbol{\chi})=\mathcal{B}(\bw,\boldsymbol{\chi}),\quad \forall\bw\in H^\div,~\boldsymbol{\chi}\in \bS_h^p\ .
\end{align}
As a consequence of (\ref{eq:femident}) and Lemma \ref{lem:approxb} we have that
\begin{equation}
\|w-R_hw\|_k\lesssim h^{s-k}\|w\|_s,\quad w\in H^s,~ 1\leq s\leq r,~ k=0,1\ ,
\end{equation}
while for $\bR_h$ we have the following error estimate (see also \cite{Thomee}):
\begin{proposition}\label{prop:3.1}
If $\bw\in \bH^{s}_0$ and $\bR_h\bw$ is the projection defined as
$$\mathcal{B}(\bR_h\bw,\bchi)=\mathcal{B}(\bw,\bchi),\quad \mbox{ for all }~ \bchi \in \bS_h^p\ ,$$
then for $2\leq s\leq p$ it holds
\begin{equation}\label{eq:e57}
\norma{\bw-\bR_h\bw}\lesssim h^{s-1}\|\bw\|_{s}\ .
\end{equation}
\end{proposition}
\begin{proof}
 For any $\bchi\in\bS_h^p$ we have
$$\norma{\bw-\bR_h\bw}\leq\norma{\bw-\bchi}+\norma{\bchi-\bR_h\bw}\ .$$
From Lemma \ref{lem:bilinear} we have
\begin{align*}
\norma{\bchi-\bR_h\bw}^2&\lesssim \mathcal{B}(\bchi-\bR_h\bw,\bchi-\bR_h\bw) \\
&\lesssim \mathcal{B}(\bchi-\bw,\bchi-\bR_h\bw) \\
&\lesssim \norma{\bchi-\bw}\norma{\bchi-\bR_h\bw}\ .
\end{align*}
Thus $\norma{\bchi-\bR_h\bw}\lesssim \norma{\bchi-\bw}$, and so, by Lemma \ref{lem:approxb} we have
$$\norma{\bw-\bR_h\bw}\lesssim \inf_{\bchi\in\bS_h^p}\norma{\bw-\bchi}\lesssim h^{s-1}\|\bw\|_{s}\ ,$$
which completes the proof.
\end{proof}

\begin{remark}
By the definition of the norm $\norma{\bw}$ for any $\bw\in \bH^1$, we have that 
\begin{equation}\label{eq:helpineq}
\|\bw\cdot\bn\|_{\partial\Omega}\lesssim h^{1/2}\norma{\bw}\ .
\end{equation}
If $\bw$ is such that $\bw\cdot\bn=0$ on $\partial \Omega$, we can see that although 
the elliptic projection does not satisfy $\bR_h\bw\cdot \bn=0$, it converges to 0 as $h\to 0$. Indeed, we have
$$\|\bR_h\bw\cdot \bn\|_{\partial\Omega}\lesssim \|(\bR_h\bw-\bw)\cdot\bn\|_{\partial\Omega}\lesssim h^{1/2}\norma{\bR_h\bw-\bw}\lesssim h^{3/2}\|\bw\|_{2}\ .$$

\end{remark}

\subsection{Standard Galerkin method for the BBM-BBM system}

We consider now the BBM-BBM system (\ref{eq:varbottopeps}) with boundary conditions $\bu\cdot\bn=0$ and $\nabla \eta\cdot \bn=0$ on $\partial \Omega$ and smooth initial conditions (\ref{eq:ics}). Without loss of generality and for economy in notation we take $\varepsilon=1$. The Galerkin finite element method semidiscretization problem is defined as follows: \\
Seek an approximate solution $(\eta^h,\bu^h)\in S_h^r\times \bS_h^p$ such that
\begin{equation}\label{eq:semidisc2}
\begin{aligned}
&\mathcal{A}(\eta^h_t,\chi)-( (D+\eta^h)\bu^h,\nabla \chi)=0,\quad \forall\chi\in S_h^r\ , \\
&\mathcal{B}(\bu^h_t,\boldsymbol{\chi})+( \nabla(\eta^h+\tfrac{1}{2}|\bu^h|^2), D^2 \boldsymbol{\chi}) =0,\quad \forall\boldsymbol{\chi}\in \bS_h^p\ ,
\end{aligned}
\end{equation}
where the symmetric bilinear forms $\mathcal{A}$ and $\mathcal{B}$ are defined as before, 
and with initial data 
$$(\eta^h_0,\bu_0^h)=(R_h\eta_0(\bx),\bR_h\bu_0(\bx))\ ,$$
where $R_h$ is the elliptic projection defined as
$$\mathcal{A}(R_h\eta_0,\chi)=\mathcal{A}(\eta_0,\chi),\quad \mbox{ for all }~\chi\in S_h^r\ ,$$
and $\bR_h$ is the elliptic projections defined in Section \ref{ellpiticproj}.

As in Section \ref{sec:wellposed} we define the functions $f_h:\bL^2\to S_h^r$ and $g_h: H^1\to \bS_h^p$ such that
\begin{equation}
\mathcal{A}(f_h(\bw),\chi)=(\bw,\nabla\chi), \quad \mbox{ for all }~ \chi\in S_h^r\ ,
\end{equation}
and
\begin{equation}\label{eq:es66}
\mathcal{B}(g_h(w),\bchi)=-(\nabla w, D^2 \bchi), \quad \mbox{ for all }~ \bchi \in \bS_h^p\ .
\end{equation}
These functionals help us to write the semidiscrete problem in the form of
a system of ordinary differential equations
\begin{equation}\label{eq:semidode}
\left\{
\begin{aligned}
\eta^h_t&=f_h((D+\eta^h)\bu^h)\ ,\\
\bu^h_t&=g_h(\eta^h+|\bu^h|^2)\ .
\end{aligned}\right.
\end{equation}
This system also enjoys favourable stability properties:
\begin{lemma}\label{lem:lem3.3}
\begin{enumerate}
\item[(i)] For any $\bw\in \bL^2$ we have the inequality 
\begin{equation}
\|f_h(\bw)\|_1\lesssim \|\bw\|\ .
\end{equation}
\item[(ii)] For $w \in H^1$, we also have
\begin{equation}
\norma{g_h(w)} \lesssim \|w\|+h^{1/2}\|w\|_{\partial\Omega}\ .
\end{equation}
\end{enumerate}
\end{lemma}
\begin{proof}
First we have
$$
\|f_h(\bw)\|_1^2 \leq \mathcal{A}(f_h(\bw),f_h(\bw))=(\bw,\nabla f_h(\bw))\lesssim \|\bw\|\|f_h(\bw)\|_1\ ,
$$
which implies that $\|f_h(\bw)\|_1\lesssim \|\bw\|$. For (ii) the situation is very similar:
\begin{equation*}
\begin{aligned}
\norma{g_h(w)}^2&\lesssim \mathcal{B}(g_h(w),g_h(w))
\lesssim |(\nabla w, g_h(w))|
\lesssim |(w,\Div g_h(w))|+|\langle w, g_h(w)\cdot \bn\rangle|\\
&\lesssim \|w\|\|g_h(w)\|_\div+\|w\|_{\partial\Omega}\|g_h(w)\cdot\bn\|_{\partial\Omega}\\
&\lesssim \|w\|\norma{g_h(w)}+\|w\|_{\partial\Omega}h^{1/2}(h^{-1/2}\|g_h(w)\cdot\bn\|_{\partial\Omega})\\
&\lesssim (\|w\|+h^{1/2}\|w\|_{\partial\Omega}) \norma{g_h(w)}\ ,
\end{aligned}
\end{equation*}
which implies the desired estimate.
\end{proof}

\begin{lemma}\label{lem:lem3.4}
For $w \in H^1$ we have
\begin{equation}
\|g_h(w)\|_{\div}\lesssim h\|w\|_1+\|w\|\ .
\end{equation}
\end{lemma}
\begin{proof}
Let $\bv$ be the solution of the system $L\bv=\nabla w$ with boundary condition $\bv\cdot\bn=0$, where the operator $L:H^\div_{1,0}\to L^2$ is such that $L\bv=\bv-\frac{1}{6}\nabla(\Div (D^2\bv))$. Assuming that  we have
\begin{equation*}
\|\nabla w\|_{-\div} =
\sup_{\substack{\bz\in H^\div_0 \\ \bz \ne 0}}\frac{(\nabla w,\bz)}{\|\bz\|_\div}
=\sup_{\substack{\bz\in H^\div_0 \\ \bz \ne 0 }}\frac{(L\bv,\bz)}{\|\bz\|_\div}\geq \frac{(L\bv,\bv)}{\|\bv\|_\div}
\geq C\frac{\|\bv\|_\div^2}{\|\bv\|_\div}\ ,
\end{equation*}
and thus $\|\bv\|_\div\lesssim \|\nabla w\|_{-\div}$. Moreover, for $\bz\in H^\div_0$ we have
\begin{equation*}
\frac{(\nabla w,\bz)}{\|\bz\|_\div}=-\frac{(w,\Div\bz)}{\|\bz\|_\div} \lesssim \frac{\|w\|\|\bz\|_\div}{\|\bz\|_\div}\ .
\end{equation*}
Thus, we conclude that $\|\bv\|_\div\lesssim \|w\|$, and therefore we have
$$\|\bR_h\bv-\bv\|_\div\lesssim h\|\bv\|_{2}\lesssim h\|\bv\|_{1,\div} \lesssim  h\| w\|_1\ .$$
This is also due to the fact that $\Curl \bv=0$, which implies $\|\bv\|_2\lesssim \|\bv\|_{1,\div}$.
Observing now that $g_h(w)=-\bR_h\bv$ we have
$$\|g_h(w)\|_\div=\|\bR_h\bv\|_{\div}\lesssim \|\bR_h\bv-\bv\|_{\div}+\|\bv\|_\div\lesssim h\|w\|_1+\|w\|\ ,$$
which completes the proof.
\end{proof}


\subsection{Error estimates} Here we study the convergence of the numerical solution to the exact solution and we estimate the errors in standard norms. Specifically, we have the following theorem:
\begin{theorem}\label{thrm:t3.1}
Under the conditions of Theorem \ref{thrm:main2} and for any $T< T_\text{max}$ where $T_\text{max}$ is the maximal time of existence of the sufficiently smooth solution $(\eta,\bu)$ of the continuous problem, there exists $h_0$ such that for any $h\in(0,h_0)$ and $r\geq2$, $p\geq 3$ the semidiscrete problem (\ref{eq:semidisc2}), has a unique solution $(\eta^h,\bu^h)\in S_h^r\times \bS_h^{p}$  in the interval $[0,T]$. Moreover,  there exists a constant $C=C(\eta,\bu,T)$ independent of $h$ such that
\begin{equation}\label{eq:errest1}
\|\eta^h-\eta\|+\|\bu^h-\bu\|_\div \leq C(h^{r}+h^{p-1})\ ,
\end{equation}
and 
 \begin{equation}\label{eq:errest2}
\|\eta^h-\eta\|_1+\|\bu^h-\bu\|_\div \leq C(h^{r-1}+h^{p-1})\ ,
\end{equation}
for all $t\in [0,T]$.
\end{theorem}
\begin{proof}
First of all, assume that there is a positive constant $M$,  independent of time, such that $\max(\|\eta\|_{\infty},\|\bu\|_{1,\infty}) \leq M/2$ for all $t\in [0,T]$.
Then, the initial conditions $\eta^h_0$ and $\bu^h_0$ are appropriately bounded.
In particular, for $h$ small enough we have that
$$\|\eta^h_0\|_{\infty}\leq \|\eta^h_0-\eta_0\|_{\infty}+\|\eta_0\|_{\infty}=\|R_h\eta_0-\eta_0\|_{\infty}+\|\eta_0\|_{\infty}\leq C\gamma_r(h)+\|\eta_0\|_{\infty}\leq M\ ,$$
where $\gamma_r(h)=h^r|\log h|^{\bar{r}}$ with $\bar{r}=0$ if $r>2$ and $\bar{r}=1$ when $r=2$, \cite{Scott1976}.
Similarly, considering the elliptic projection $\tilde{\bR}_h\bu=(R_h u,R_h v)$ for any $\bu=(u,v)$ sufficiently smooth we have that 
$$
\begin{aligned}
\|\bu_0^h\|_\infty &\leq \|\bu_0^h-\bu_0\|_\infty+\|\bu_0\|_\infty\\
&\leq \|\bR_h\bu_0-\tilde{\bR}_h\bu_0\|_\infty +\|\tilde{\bR}_h\bu_0-\bu_0\|_\infty+M/2\\
&\leq Ch^{-1}\|\bR_h\bu_0-\tilde{\bR}_h\bu_0\|+C\gamma_p(h)+M/2\\
&\leq Ch^{-1}(\|\bR_h\bu_0-\bu_0\|+\|\bu_0-\tilde{\bR}_h\bu_0\|)+C\gamma_p(h)+M/2\\
&\leq Ch^{p-2}+C\gamma_p(h)+M/2\ ,
\end{aligned}
$$
and thus for sufficiently small $h$ we have $\|\bu^h_0\|_{\infty}\leq M$.


Moreover, it is easily seen that the semidiscrete system of ordinary differential equations (\ref{eq:semidode}) has a unique solution for at least a small time interval $[0,t_h]$. This is because $f_h$ and $g_h$ are Lipschitz functions for $\|\eta^h\|_\infty\leq M$ and $\|\bu^h\|_{\infty}\leq M$ for fixed $h>0$ due to Lemma \ref{lem:lem3.3}. 
  Thus, we assume that there is a maximal time $t_h\in[0,T]$ such that
$\|\eta^h\|_{\infty}\leq M$ and $\|\bu^h\|_{\infty}\leq M$ for all $t\leq t_h$. For the same time interval of the existence of the semidiscrete solution we can also assume that $\|\bu^h\|_{1,\infty}\leq Ch^{p-3}+M/2<\infty$. Thus, $\bu^h\in W^{1,\infty}(\Omega)\times W^{1,\infty}(\Omega)$ for sufficiently small $h>0$, and thus the trace inequality $\|\bu^h\|_{\infty,\partial\Omega}\lesssim \|\bu^h\|_\infty$ holds true.

We consider the quantities 
$$\begin{aligned}
&\theta=\eta^h-R_h\eta,\qquad  \rho=R_h\eta-\eta, \\
&\bzeta=\bu^h-\bR_h\bu,\qquad  \bxi=\bR_h\bu-\bu\ .
\end{aligned}
$$
From the approximation properties of the elliptic projection, see Lemma \ref{lem:approxb}, we have $\|\rho\|\lesssim h^r$ and $\|\bxi\|_\div\lesssim h^{p-1}$. Then, the errors are defined as 
$$e=\eta^h-\eta=\theta+\rho, \qquad \be=\bu^h-\bu=\bzeta+\bxi\ .$$
We observe that
$$\begin{aligned}
\mathcal{A}(\theta_t,\chi)&= \mathcal{A}(\eta^h_t,\chi)-\mathcal{A}(\eta_t,\chi)\\
&= ((D+\eta^h)\bu^h,\nabla\chi)-((D+\eta)\bu,\nabla\chi) \\
&= \mathcal{A}(f_h((D+\eta^h)\bu^h-(D+\eta)\bu),\chi)\ ,
\end{aligned}
$$
and since this is true for all $\chi\in S_h$ we have that
$$\theta_t=f_h((D+\eta^h)\bu^h-(D+\eta)\bu)\ .$$
Rearranging the terms in the last expression we have
$$\theta_t=f_h(D(\bzeta+\bxi))+f_h(\eta^h(\bzeta+\bxi))+f_h((\theta+\rho)\bu)\ .$$
Therefore, using Lemma \ref{lem:lem3.3} we have
$$
\|\theta_t\|_1 \lesssim (1+\|\eta^h\|_{\infty})(\|\bzeta\|+\|\bxi\|)+\|\bu\|_{\infty}(\|\theta\|+\|\rho\|)\ ,
$$
which implies
\begin{equation}\label{eq:theta1}
\|\theta_t\|_1\lesssim (h^{r}+h^{p-1})+\|\theta\|+\|\bzeta\|\ .
\end{equation}
Similarly, for any $\bchi\in\bS_h^p$ and by the definition of $\mathcal{B}$ we have
$$\begin{aligned}
\mathcal{B}(\bu_t,\bchi)&=(D\bu_t,D\bchi)+\frac{1}{6}(\Div (D^2\bu_t),\Div(D^2\bchi))-\frac{1}{6}\langle \Div (D^2\bu_t),D^2\bchi\cdot \bn\rangle\\ 
&\qquad -\frac{1}{6}\langle D^2\bu_t\cdot \bn,\Div (D^2\bchi)\rangle+\frac{\gamma}{h}\langle D^2\bu_t\cdot\bn,\bchi\cdot\bn\rangle\\
\text{(since $\bu_t\cdot \bn=0$)} &= (D\bu_t,D\bchi)+\frac{1}{6}(\Div (D^2\bu_t),\Div(D^2\bchi))-\frac{1}{6}\langle \Div (D^2\bu_t),D^2\bchi\cdot \bn\rangle\\
(\text{divergence thrm})&= (D^2\bu_t,\bchi)-\frac{1}{6}(\nabla\Div (D^2\bu_t),D^2\bchi)+\frac{1}{6}\langle \Div(D^2\bu_t),D^2\bchi\cdot\bn\rangle-\frac{1}{6}\langle \Div (D^2\bu_t),D^2\bchi\cdot \bn\rangle\\
&=(D^2(\bu_t-\frac{1}{6}\nabla\Div (D^2\bu_t)),\bchi)\\
&=-(\nabla(\eta+\tfrac{1}{2}|\bu|^2),D^2\bchi)\\
&=\mathcal{B}(g_h(\eta+\tfrac{1}{2}|\bu|^2),\bchi)\ .
\end{aligned}$$
Note also that
$$
\begin{aligned}
\mathcal{B}(\bzeta_t,\bchi)&=\mathcal{B}(\bu^h_t,\bchi)-\mathcal{B}(\bu_t,\bchi)\\
&=\mathcal{B}(g_h(\eta^h+\tfrac{1}{2}|\bu^h|^2-\eta-\tfrac{1}{2}|\bu|^2),\bchi)\ .
\end{aligned}
$$
Therefore, we can write $\bzeta_t$ as
$$\begin{aligned}
\bzeta_t&= g_h(\eta^h+\tfrac{1}{2}|\bu^h|^2-\eta-\tfrac{1}{2}|\bu|^2)\\
&=g_h(\theta+\rho)+g_h((\bzeta+\bxi)\cdot \bu^h)+g_h(\bu\cdot(\bzeta+\bxi))\ .\\
\end{aligned}$$
Using again Lemma \ref{lem:lem3.3} we have
$$
\begin{aligned}
\|\bzeta_t\|_\div &\lesssim \|\theta\|+\|\rho\|+h(\|\theta\|_1+\|\rho\|_1)+(\|\bu^h\|_{\infty}+\|\bu\|_{\infty})(\|\bzeta\|+\|\bxi\|)\\
&\quad +h^{1/2}\|\bzeta\cdot\bu^h\|_{\partial\Omega}
+h^{1/2}\|\bxi\cdot\bu^h\|_{\partial\Omega}
+h^{1/2}\|\bzeta\cdot\bu\|_{\partial\Omega}
+h^{1/2}\|\bxi\cdot\bu\|_{\partial\Omega}\\
&\lesssim (h^r+h^{p-1}) + \|\theta\|+ h^{1/2}[\|\bu^h\|_{\infty,\partial\Omega}(\|\bzeta\|_{\partial\Omega}+\|\bxi\|_{\partial\Omega})+\|\bu\|_{\infty,\partial\Omega}(\|\bzeta\|_{\partial\Omega}+\|\bxi\|_{\partial\Omega})]\\
&\lesssim (h^r+h^{p-1})+ \|\theta\|+ h^{1/2}[\|\bu^h\|_\infty(h^{-1/2}\|\bzeta\|+h^{1/2}h^{-1/2}\|\bxi\cdot\bn\|_{\partial\Omega})\\
&\quad +\|\bu\|_\infty(h^{-1/2}\|\bzeta\|+h^{1/2}h^{-1/2}\|\bxi\cdot\bn\|_{\partial\Omega})]\\
&\lesssim (h^r+h^{p-1})+\|\theta\|+ [\|\bu^h\|_\infty(\|\bzeta\|+h\norma{\bxi})+\|\bu\|_\infty(\|\bzeta\|+h\norma{\bxi})]\\
\end{aligned}
$$
which after applying Proposition \ref{prop:3.1} we obtain the estimate
\begin{equation}\label{eq:zeta1}
\|\bzeta_t\|_\div \lesssim (h^{r}+h^{p-1})+\|\theta\|+\|\bzeta\|\ .
\end{equation}
Finally, adding (\ref{eq:theta1}) and (\ref{eq:zeta1}) we obtain
$$\frac{d}{dt}\left(\|\theta\|^2_1+\|\bzeta\|_\div^2 \right)\lesssim (h^{r}+h^{p-1})^2 + \|\theta\|^2_1+\|\bzeta\|_\div^2\ ,$$
from which, using the Gronwall inequality we obtain the following superconvergence result for $0<t\leq t_h$:
\begin{equation}
\|\theta\|_1+\|\bzeta\|_\div\lesssim h^r+h^{p-1}\ .
\end{equation}
 The error estimate then follows from the fact
$$\|e\|+\|\be\|_\div\leq \|\theta\|+\|\rho\|+\|\bzeta\|_\div+\|\bxi\|_\div\lesssim h^r+h^{p-1} \ ,$$
and
$$\|e\|_1+\|\be\|_\div\leq \|\theta\|_1+\|\rho\|_1+\|\bzeta\|_\div+\|\bxi\|_\div\lesssim h^{r-1}+h^{p-1} \ .$$
Having the convergence until $t_h$, we can show that the solution is indeed bounded in the appropriate norms for sufficiently small $h$. More precisely, we have
$$
\begin{aligned}
\|\eta^h\|_{\infty}&\leq \|\eta^h-R_h\eta\|_{\infty}+\|R_h\eta-\eta\|_{\infty}+\|\eta\|_{\infty}\\
&\leq Ch^{-1}\|\eta^h-R_h\eta\|+\gamma_r(h)+M \leq Ch^{r-1}+M/2 < M\ .
\end{aligned}
$$
Similarly, for sufficiently small $h$ we show again that $\bu^h\in \bL^\infty$:
$$
\begin{aligned}
\|\bu^h\|_\infty &\leq \|\bu^h-\bu\|_\infty+\|\bu\|_\infty\\
&\leq \|\bu^h-\bR_h\bu\|_\infty +\|\bR_h\bu-\bu\|_\infty+M/2\\
&\leq \|\bu^h-\bR_h\bu\|_\infty +\|\bR_h\bu-\tilde{\bR}_h\bu\|_\infty+\|\tilde{\bR}_h\bu-\bu\|_\infty+M/2\\
&\leq Ch^{p-2}+C\gamma_p(h)+M/2\leq M\ .
\end{aligned}
$$
These estimates contradict the assumption of the existence of a maximal time $t_h$, and thus we conclude using the bootstrap theorem (cf. \cite{tao2006}) that $t_h=T$. 
\end{proof} 

\begin{remark}
From the proof of Theorem \ref{thrm:t3.1} we observe that the convergence of the semi-discrete solution in the $L^\infty$-norm is also established in the case $r\geq 2$ and $p\geq 3$. When $p=2$, the time $t_h$ cannot be extended up to $T$, although the error estimates (\ref{eq:errest1})--(\ref{eq:errest2}) are still valid. It is worth mentioning that we did not experience any problems when we tested 
the case $p=2$ numerically, and the results were always stable for the timescales we used.
\end{remark}

\begin{remark}
Using (\ref{eq:helpineq}) and (\ref{eq:neqinvers}) we observe that
$$
\begin{aligned}
\|\bu^h\cdot \bn\|_{\partial\Omega}&=\|(\bu^h-\bu)\cdot \bn\|_{\partial\Omega} 
\leq \|\bzeta\cdot\bn\|_{\partial\Omega} +\|\bxi\cdot\bn\|_{\partial\Omega}\\
&\leq h^{1/2}(\norma{\bzeta} +\norma{\bxi}) 
\lesssim h^{1/2}(h^{-1/2}\|\bzeta\|_\div+h^{p-1}) \\
&\lesssim h^r+h^{p-1}\ ,
\end{aligned}$$
thus, the normal trace of the numerical solution $\bu^h\cdot\bn$ converges to zero as $h\to0$. Experimentally, we found out that this estimate is not sharp enough and that $\|\bu^h\cdot \bn\|_{\partial\Omega}$ converges to zero even faster following an undetermined superconvergence law.
\end{remark}

\begin{remark}
The error estimate (\ref{eq:errest1}) appears to be sharp as we confirm experimentally in the next section. In particular, we verify that the error estimate in the case $(r,p)=(2,3)$ is $\|e\|+\|\be\|=O(h^r+h^{p-1})=O(h^2)$.
\end{remark}

\section{Numerical experiments}\label{sec:numerics}
In what follows we perform a series of numerical experiments with the aim of validating 
the new model for the generation and propagation of shallow water waves. 
First we present an experimental validation of the convergence rates analyzed 
in Section \ref{sec:numsols} for the semidiscrete problem (\ref{eq:semidisc}). 
For this purpose we implement the time-discretization with the classical, explicit
four-stage, fourth-order Runge-Kutta scheme and which has been analyzed and used extensively 
in similar problems where the regularization terms result 
into a non-stiff system of ordinary differential equations \cite{DMS2010, DMS2007, KMS2019}.
\subsection{Numerical confirmation of convergence rates in a two-dimensional domain with non-trivial bathymetry}
Our first task is the numerical verification of the error estimates (\ref{eq:errest1}) 
and (\ref{eq:errest2}). 
For this purpose we consider the initial-boundary value problem (\ref{eq:bbm2a})--(\ref{eq:bcs})
in the domain $\Omega=[0,1]\times [0,1]$, equipped with an appropriate forcing term so that 
the resulting system admits the following the functions as an exact solution:
\begin{equation}
\begin{aligned}
\eta(\bx,t)&=e^{t}\cos(\pi x)\cos(\pi y)\ ,\\
\bu(\bx,t)&= e^{t}\left(\cos(\pi y)\sin(\pi x),\cos(\pi x)\sin(\pi y)\right)\ .
\end{aligned}
\end{equation}
This specific exact solution satisfies the boundary conditions $\nabla \eta\cdot\bn=0$ and $\bu\cdot \bn=0$, and also the condition $\Curl \bu(\bx,t)=0$ for all $t\geq 0$, and therefore complies with the theory developed in the previous sections.
The bottom topography is chosen to be 
$$D(\bx)=1-10^{-2}e^{-|\bx|^2}\ .$$
We further consider regular, uniform triangulations $\mathcal{T}_h$ 
of $\Omega$ for $h=h_i=1/N$, $N=8+4i$, for $i=0,1,\cdots, 6$. 
For each grid $\mathcal{T}_h$ we integrate the system (\ref{eq:semidisc}) 
up to time $T=1$ using the classical, explicit four-stage, fourth-order Runge-Kutta method 
with stepsize $\Delta t=5\times 10^{-4}$ to ensure that errors induced by the time integration 
are negligible compared to the respective errors of the spatial discretization.  
The error of the Runge-Kutta method is expected to be of the order of $(\Delta t)^4$, 
while as we saw in the previous section the errors from the semidiscretization appear 
to have smaller convergence rates, especially the cases we consider here, 
which are linear, quadratic and cubic Lagrange elements. 
During the time integration we recorded the numerical errors 
$E^0(\eta)=\|e\|$, $E^0(\bu)=\|\be\|$, $E^1(\eta)=\|e\|_1$ and $E^{\div}(\bu)=\|\be\|_\div$, 
and we compute the experimental convergence rates $R$ defined as
$$R^\alpha_i=\log\left(\frac{E^\alpha_i}{E^\alpha_{i+1}}\right)/\log\left(\frac{h_i}{h_{i+1}}\right),~ i=0,1,\cdots, 6\ ,$$
where $\alpha$ is $0$, $1$ or $\div$. 
It is noted that for the penalty parameter of Nitsche's method we used $\gamma=1000$. This value was the largest value greater than $10$ we tried and worked well. We didn't observe any instabilities for the values we tried, while in some cases (depending on the choice of the bottom topography) the inversion the regularization operator was more accurate for smaller values of $\gamma$. Moreover, for implementation purposes, instead of using the bottom topography $D(\bx)$ 
we use the $L^2$-projection of the depth function into the space $S_h^r$. 

First we start with the case $p=r+1$ where convergence is guaranteed by Theorem \ref{thrm:t3.1}. 
In Tables \ref{tab:tab1a}, \ref{tab:tab1b} we present the errors and the convergence rates in the case 
where $(r,p)=(2,3)$. The specific experiment confirms the optimal error estimate (\ref{eq:errest1}) 
for the $L^2$-norm of $\eta$ and $H^\div$-norm of $\bu$. 
The error between $\bu^h$ and $\bu$ in the $L^2$-norm apparently converges to $0$ 
with the same rate as in the $H^\div$-norm which again is a confirmation of Theorem \ref{thrm:t3.1}.  
The convergence rates for both $\bu^h$ and $\eta^h$ in Theorem \ref{thrm:t3.1} are optimal, 
but they do not guarantee optimal convergence rates in other norms 
except for the optimal convergence rate for the $H^1$-norm of the error in $\eta$. 
An interesting observation derived from the specific numerical experiment 
is that the errors in the $L^2$-norm for both $\eta$ and $\bu$ are of the same order. 
On the other hand, the respective errors based on the $H^1$-norm appear to have different orders. 
The error $\|\bu-\bu^h\|_1=O(10^{-3})$ while $\|\eta-\eta^h\|_1=O(10^{-1})$, 
perhaps due to the use of quadratic polynomials for $\bu$ and linear polynomials in $\eta$.

Very similar results can be observed in the case where $(r,p)=(3,4)$ in Table \ref{tab:tab2a} and \ref{tab:tab2b} with the exception that the convergence rates based on the $L^2$-norm are all optimal this time. This phenomenon is due to the specific choice of the bottom topography. For different bottom topography $D(x,y)=-1/20(x+y)+3/2$ we observe suboptimal $L^2$-norm based rates for $\eta$ again. For the specific linear bottom the $H^1$-norm based convergence rates for $\bu$ appears also to be suboptimal. Therefore, the only error estimate that can be confirmed is the one proved in Theorem \ref{thrm:t3.1}. 
\begin{table}[htbp]
  \centering
  \caption{Convergence rates for the unknown $\bu$ for the spatial discretization in terms of the 
           maximum side length of the triangulation by $h$ for the case $(r,p)=(2,3)$.}
    \begin{tabular}{c|cc|cc|cc}\hline
   $h$ & $\|\bu-\bu_h\|$ &  $R^0_i$ &  $\|\bu-\bu_h\|_\div$ &  $R^\div_i$ & $\|\bu-\bu_h\|_1$ & $R^1_i$ \\ \hline  
  $1.250\times 10^{-1} $ &  $2.704\times 10^{-3}$ &               &  $2.146\times 10^{-2}$ &                &  $2.397\times 10^{-2}$ &   \\   
  $8.333\times 10^{-2} $ &  $1.200\times 10^{-3}$ &  $2.003$ & $9.419\times 10^{-3}$ &  $2.031$ &  $1.054\times 10^{-2}$ &  $2.027$ \\ 
  $6.250\times 10^{-2} $ &  $6.748\times 10^{-4}$ &  $2.001$ & $5.269\times 10^{-3}$ &  $2.019$ &  $5.920\times 10^{-3}$ &  $2.004$  \\ 
  $5.000\times 10^{-2} $ &  $4.318\times 10^{-4}$ &  $2.001$ & $3.363\times 10^{-3}$ &  $2.012$ &  $3.803\times 10^{-3}$ &  $1.984$ \\ 
  $4.167\times 10^{-2} $ &  $2.998\times 10^{-4}$ &  $2.001$ & $2.332\times 10^{-3}$ &  $2.007$ &  $2.659\times 10^{-3}$ &  $1.962$ \\ 
  $3.571\times 10^{-2} $ &  $2.203\times 10^{-4}$ &  $2.000$ & $1.713\times 10^{-3}$ &  $2.002$ &  $1.972\times 10^{-3}$ &  $1.940$ \\ 
  $3.125\times 10^{-2} $ &  $1.686\times 10^{-4}$ &  $2.000$ & $1.312\times 10^{-3}$ &  $1.998$ &  $1.527\times 10^{-3}$ &  $1.915$ \\ \hline  
\end{tabular}
  \label{tab:tab1a}%
\end{table}%
\begin{table}[htbp]
  \centering
  \caption{Convergence rates for the unknown $\eta$ for the spatial discretization in terms of the 
           maximum side length of the triangulation by $h$ for the case $(r,p)=(2,3)$.}
    \begin{tabular}{c|cc|cc}\hline
   $h$ & $\|\eta-\eta_h\|$ &  $R^0_i$ &  $\|\eta-\eta_h\|_1$ &  $R^1_i$ \\ \hline  
  $1.250\times 10^{-1} $ & $ 1.021\times 10^{-2} $ &                &  $6.276\times 10^{-1} $ &      \\  
  $8.333\times 10^{-2} $ & $ 4.510\times 10^{-3} $ &  $2.014$ &  $4.179\times 10^{-1} $ &  $1.003$  \\  
  $6.250\times 10^{-2} $ & $ 2.532\times 10^{-3} $ &  $2.007$ &  $3.133\times 10^{-1} $ &  $1.001$  \\  
  $5.000\times 10^{-2} $ & $ 1.619\times 10^{-3} $ &  $2.004$ &  $2.506\times 10^{-1} $ &  $1.001$  \\  
  $4.167\times 10^{-2} $ & $ 1.124\times 10^{-3} $ &  $2.003$ &  $2.088\times 10^{-1} $ &  $1.001$  \\  
  $3.571\times 10^{-2} $ & $ 8.253\times 10^{-4} $ &  $2.002$ &  $1.790\times 10^{-1} $ &  $1.000$ \\ 
  $3.125\times 10^{-2} $ & $ 6.317\times 10^{-4} $ &  $2.001$ &  $1.566\times 10^{-1} $ &  $1.000$  \\ \hline  
\end{tabular}
  \label{tab:tab1b}%
\end{table}%
\begin{table}[htbp]
  \centering
  \caption{Convergence rates for the unknown $\bu$ for the spatial discretization in terms of the 
           maximum side length of the triangulation by $h$ for the case $(r,p)=(3,4)$.}
    \begin{tabular}{c|cc|cc|cc}\hline
   $h$ & $\|\bu-\bu_h\|$ &  $R^0_i$ &  $\|\bu-\bu_h\|_\div$ &  $R^\div_i$ & $\|\bu-\bu_h\|_1$ & $R^1_i$ \\ \hline  
  $1.250\times 10^{-1} $ &  $2.248\times 10^{-5}$ &               &  $1.456\times 10^{-3}$ &                &  $1.491\times 10^{-3}$ &   \\   
  $8.333\times 10^{-2} $ &  $4.510\times 10^{-6}$ &  $3.962$ & $4.378\times 10^{-4}$ &  $2.964$ &  $4.489\times 10^{-4}$ &  $2.961$ \\ 
  $6.250\times 10^{-2} $ &  $1.437\times 10^{-6}$ &  $3.977$ & $1.857\times 10^{-4}$ &  $2.982$ & $ 1.907\times 10^{-4}$ &  $2.975$  \\ 
  $5.000\times 10^{-2} $ &  $5.911\times 10^{-7}$ &  $3.980$ & $9.530\times 10^{-5}$ &  $2.989$ &  $9.815\times 10^{-5}$ &  $2.978$ \\ 
  $4.167\times 10^{-2} $ &  $2.861\times 10^{-7}$ &  $3.980$ & $5.523\times 10^{-5}$ &  $2.992$ &  $5.705\times 10^{-5}$ &  $2.976$ \\ 
  $3.571\times 10^{-2} $ &  $1.550\times 10^{-7}$ &  $3.977$ & $3.482\times 10^{-5}$ &  $2.993$ &  $3.609\times 10^{-5}$ &  $2.971$ \\ 
  $3.125\times 10^{-2} $ &  $9.119\times 10^{-8}$ &  $3.972$ & $2.334\times 10^{-5}$ &  $2.994$ &  $2.429\times 10^{-5}$ &  $2.965$ \\ \hline  
\end{tabular}
  \label{tab:tab2a}%
\end{table}%
\begin{table}[htbp]
  \centering
  \caption{Convergence rates for the unknown $\eta$ for the spatial discretization in terms of the 
           maximum side length of the triangulation by $h$ for the case $(r,p)=(3,4)$.}
    \begin{tabular}{c|cc|cc}\hline
   $h$ & $\|\eta-\eta_h\|$ &  $R^0_i$ &  $\|\eta-\eta_h\|_1$ &  $R^1_i$ \\ \hline  
  $1.250\times 10^{-1} $ & $ 4.400\times 10^{-4} $ &                &  $3.203\times 10^{-2} $ &      \\  
  $8.333\times 10^{-2} $ & $ 1.318\times 10^{-4} $ &  $2.973$ &  $1.424\times 10^{-2} $ &  $2.000$  \\  
  $6.250\times 10^{-2} $ & $ 5.583\times 10^{-5} $ &  $2.986$ &  $8.008\times 10^{-3} $ &  $2.000$  \\  
  $5.000\times 10^{-2} $ & $ 2.864\times 10^{-5} $ &  $2.992$ &  $5.125\times 10^{-3} $ &  $2.000$  \\  
  $4.167\times 10^{-2} $ & $ 1.659\times 10^{-5} $ &  $2.994$ &  $3.559\times 10^{-3} $ &  $2.000$  \\  
  $3.571\times 10^{-2} $ & $ 1.045\times 10^{-5} $ &  $2.996$ &  $2.615\times 10^{-3} $ &  $2.000$ \\ 
  $3.125\times 10^{-2} $ & $ 7.006\times 10^{-6} $ &  $2.997$ &  $2.002\times 10^{-3} $ &  $2.000$  \\ \hline  
\end{tabular}
  \label{tab:tab2b}%
\end{table}%

We close this section by presenting the experimental convergence rates when $r=p$. Tables \ref{tab:tab3a} and \ref{tab:tab3b} presents the errors and the convergence rates for $r=p=2$. In this case we obtained optimal convergence rates in all norms. In Tables \ref{tab:tab4a} and \ref{tab:tab4b} we present the respective errors and convergence rates for the case $r=p=3$. In this case again it is quite obvious that there is no optimal convergence in $\bL^2$ and $\bH^1$ norms for the solution $\bu$ as the rate is decreasing steadily. On the other hand the convergence rate in $H^\div$-norm is optimal again for $\bu$ and also the $L^2$ and $H^1$ convergence rates for $\eta$ are also optimal. 

\begin{table}[htbp]
  \centering
  \caption{Convergence rates for the unknown $\bu$ for the spatial discretization in terms of the 
           maximum side length of the triangulation by $h$ for the case $(r,p)=(2,2)$.}
    \begin{tabular}{c|cc|cc|cc}\hline
   $h$ & $\|\bu-\bu_h\|$ &  $R^0_i$ &  $\|\bu-\bu_h\|_\div$ &  $R^\div_i$ & $\|\bu-\bu_h\|_1$ & $R^1_i$ \\ \hline  
  $1.250\times 10^{-1} $ &  $1.792\times 10^{-2}$ &               &  $4.166\times 10^{-1}$ &                &  $5.933\times 10^{-1}$ &   \\   
  $8.333\times 10^{-2} $ &  $7.972\times 10^{-3}$ &  $1.998$ & $2.771\times 10^{-1}$ &  $1.005$ &  $3.931\times 10^{-1}$ &  $1.015$ \\ 
  $6.250\times 10^{-2} $ &  $4.486\times 10^{-3}$ &  $1.999$ & $2.077\times 10^{-1}$ &  $1.003$ & $ 2.942\times 10^{-1}$ &  $1.007$  \\ 
  $5.000\times 10^{-2} $ &  $2.871\times 10^{-3}$ &  $1.999$ & $1.661\times 10^{-1}$ &  $1.002$ &  $2.351\times 10^{-1}$ &  $1.004$ \\ 
  $4.167\times 10^{-2} $ &  $1.994\times 10^{-3}$ &  $2.000$ & $1.384\times 10^{-1}$ &  $1.001$ &  $1.958\times 10^{-1}$ &  $1.003$ \\ 
  $3.571\times 10^{-2} $ &  $1.465\times 10^{-3}$ &  $2.000$ & $1.186\times 10^{-1}$ &  $1.001$ &  $1.678\times 10^{-1}$ &  $1.002$ \\ 
  $3.125\times 10^{-2} $ &  $1.122\times 10^{-3}$ &  $2.000$ & $1.038\times 10^{-1}$ &  $1.001$ &  $1.468\times 10^{-1}$ &  $1.001$ \\ \hline  
\end{tabular}
  \label{tab:tab3a}%
\end{table}%
\begin{table}[htbp]
  \centering
  \caption{Convergence rates for the unknown $\eta$ for the spatial discretization in terms of the 
           maximum side length of the triangulation by $h$ for the case $(r,p)=(2,2)$.}
    \begin{tabular}{c|cc|cc}\hline
   $h$ & $\|\eta-\eta_h\|$ &  $R^0_i$ &  $\|\eta-\eta_h\|_1$ &  $R^1_i$ \\ \hline  
  $1.068\times 10^{-2} $ & $ 9.135\times 10^{-3} $ &                &  $6.289\times 10^{-1} $ &      \\  
  $4.728\times 10^{-3} $ & $ 4.039\times 10^{-3} $ &  $2.013$ &  $4.183\times 10^{-1} $ &  $1.006$  \\  
  $2.656\times 10^{-3} $ & $ 2.268\times 10^{-3} $ &  $2.006$ &  $3.134\times 10^{-1} $ &  $1.003$  \\  
  $1.699\times 10^{-3} $ & $ 1.450\times 10^{-3} $ &  $2.004$ &  $2.507\times 10^{-1} $ &  $1.002$  \\  
  $1.179\times 10^{-3} $ & $ 1.007\times 10^{-3} $ &  $2.003$ &  $2.088\times 10^{-1} $ &  $1.001$  \\  
  $8.662\times 10^{-4} $ & $ 7.394\times 10^{-4} $ &  $2.002$ &  $1.790\times 10^{-1} $ &  $1.001$ \\ 
  $6.631\times 10^{-4} $ & $ 5.660\times 10^{-4} $ &  $2.001$ &  $1.566\times 10^{-1} $ &  $1.001$  \\ \hline  
\end{tabular}
  \label{tab:tab3b}%
\end{table}%
\begin{table}[htbp]
  \centering
  \caption{Convergence rates for the unknown $\bu$ for the spatial discretization in terms of the 
           maximum side length of the triangulation by $h$ for the case $(r,p)=(3,3)$.}
    \begin{tabular}{c|cc|cc|cc}\hline
   $h$ & $\|\bu-\bu_h\|$ &  $R^0_i$ &  $\|\bu-\bu_h\|_\div$ &  $R^\div_i$ & $\|\bu-\bu_h\|_1$ & $R^1_i$ \\ \hline  
  $1.250\times 10^{-1} $ &  $6.342\times 10^{-4}$ &               &  $2.409\times 10^{-2}$ &                &  $4.225\times 10^{-2}$ &   \\   
  $8.333\times 10^{-2} $ &  $1.914\times 10^{-4}$ &  $2.955$ & $1.048\times 10^{-2}$ &  $2.052$ &  $1.853\times 10^{-2}$ &  $2.033$ \\ 
  $6.250\times 10^{-2} $ &  $8.140\times 10^{-5}$ &  $2.971$ & $5.836\times 10^{-3}$ &  $2.036$ &  $1.038\times 10^{-2}$ &  $2.012$  \\ 
  $5.000\times 10^{-2} $ &  $4.192\times 10^{-5}$ &  $2.974$ & $3.713\times 10^{-3}$ &  $2.027$ &  $6.655\times 10^{-3}$ &  $1.994$ \\ 
  $4.167\times 10^{-2} $ &  $2.438\times 10^{-5}$ &  $2.972$ & $2.569\times 10^{-3}$ &  $2.021$ &  $4.641\times 10^{-3}$ &  $1.977$ \\ 
  $3.571\times 10^{-2} $ &  $1.544\times 10^{-5}$ &  $2.966$ & $1.882\times 10^{-3}$ &  $2.017$ &  $3.432\times 10^{-3}$ &  $1.959$ \\ 
  $3.125\times 10^{-2} $ &  $1.040\times 10^{-5}$ &  $2.958$ & $1.438\times 10^{-3}$ &  $2.014$ &  $2.648\times 10^{-3}$ &  $1.940$ \\ \hline  
\end{tabular}
  \label{tab:tab4a}%
\end{table}%
\begin{table}[htbp]
  \centering
  \caption{Convergence rates for the unknown $\eta$ for the spatial discretization in terms of the 
           maximum side length of the triangulation by $h$ for the case $(r,p)=(3,3)$.}
    \begin{tabular}{c|cc|cc}\hline
   $h$ & $\|\eta-\eta_h\|$ &  $R^0_i$ &  $\|\eta-\eta_h\|_1$ &  $R^1_i$ \\ \hline  
  $1.068\times 10^{-2} $ & $ 4.427\times 10^{-4} $ &                &  $3.197\times 10^{-2} $ &      \\  
  $4.728\times 10^{-3} $ & $ 1.322\times 10^{-4} $ &  $2.981$ &  $1.422\times 10^{-2} $ &  $1.997$  \\  
  $2.656\times 10^{-3} $ & $ 5.593\times 10^{-5} $ &  $2.991$ &  $8.004\times 10^{-3} $ &  $1.999$  \\  
  $1.699\times 10^{-3} $ & $ 2.867\times 10^{-5} $ &  $2.994$ &  $5.123\times 10^{-3} $ &  $1.999$  \\  
  $1.179\times 10^{-3} $ & $ 1.660\times 10^{-5} $ &  $2.996$ &  $3.558\times 10^{-3} $ &  $2.000$  \\  
  $8.662\times 10^{-4} $ & $ 1.046\times 10^{-5} $ &  $2.997$ &  $2.614\times 10^{-3} $ &  $2.000$ \\ 
  $6.631\times 10^{-4} $ & $ 7.009\times 10^{-6} $ &  $2.998$ &  $2.002\times 10^{-3} $ &  $2.000$  \\ \hline  
\end{tabular}
  \label{tab:tab4b}%
\end{table}%

Repeating the same experiments but using different bottom topographies we obtained similar results. 
In all cases investigated, we always obtained the optimal convergence rates guaranteed by Theorem \ref{thrm:t3.1}. 
For similar studies related to Boussinesq-Peregrine type system with similar regularization operators 
and the application of Nitsche's method we refer to \cite{KMS2019}. 
It is also noted that testing other initial conditions that didn't satisfy the condition $\Curl \bu=0$ 
we obtained very similar results to those presented here. 

The smooth bottom variations assumption in practice is not a major limitation on the range of validity of the model. 
The main reason is that the model is derived under the long wave assumption and it is known that bottom variations 
are not crucial for long waves of small amplitude. The shape and regularity of the boundary of 
$\Omega$ seems to be the only limitation as the use of non-convex or non-simply connected domains 
cannot be supported by the theory. 
On the other hand, in experiments with non-convex domains 
no significant or unexpected anomalies were observed (see \cite{KMS2019}).

\subsection{Experimental validation in a two-dimensional domain with uneven bottom}

In this section we present two numerical experiments in order to study the shoaling of traveling waves, 
which apparently shows the influence of the bottom topography to the solution of the system at hand. 
In both cases, experimental data are available and compared to the numerical solution.  
We also compare the regularized shallow water equations (\ref{eq:bbm}) with the simplified BBM-BBM system (\ref{eq:prgndlowb}) 
written in dimensional form. Recall that the simplified BBM-BBM system contains only 
terms of maximum order $\varepsilon$ and $\sigma^2$ 
while the BBM-BBM term contains additional terms of order $\varepsilon\sigma^2$. 
The specific experiments are standard benchmarks cases, 
and have been used numerous times 
for the validation of various Boussinesq systems and numerical models \cite{WB1999,KMS2019}.  
In both experiments a rectangular basin of dimensions $[-50,20]\times [0,1]$ 
is considered for the propagation of solitary waves over a bottom which is flat in $[-50,0]$ 
and the eventual shoaling of the solitary waves on a bottom slope of $1/50$ in $[0,20]$. 
In the first experiment, the solitary wave has amplitude $0.07~m$ while in the second the amplitude is $0.12~m$. 
The free surface is recorded at three different locations considered as wave gauges: 
$(x,y) = (0.0,0.5)$, $(x,y)=(16.25,0.5)$ and $(x,y)=(17.75,0.5)$. Figure \ref{fig:f0} shows a cross 
section along $y=0$ of the physical domain and the location of the three wave-gauges drawn in red. 
In this figure the solitary wave is the one used in the second case and is presented at its initial location. 
For the numerical experiments we consider a triangulation $\mathcal{T}_h$ 
consisted with $14,402$ triangles and stepsize $\Delta t=10^{-3}$, and the Galerkin method with $(r,p)=(2,3)$. 

\begin{figure}[H]
  \centering
  \includegraphics[width=0.7\columnwidth]{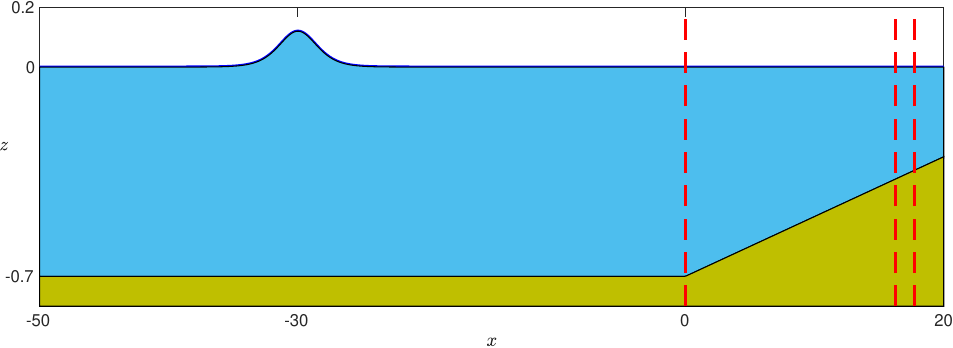}
  \caption{Cross section of the physical domain and locations of the wave-gauges}
  \label{fig:f0}
\end{figure}

Figures \ref{fig:f1} and \ref{fig:f2} present the recorded solution at the three wave gauges. 
As far as the new regularized shallow water system is concerned 
in both cases, the numerical solution is in agreement with the experimental data,
and this finding allows us to conclude that the assumption of smooth bottom variations 
is not a problem in practice for bottom topographies with slopes. 
On the other hand, the simplified BBM-BBM system (\ref{eq:prgndlowb}) fails to predict well 
the shoaling of the solitary waves. 
It is noted that we used the same initial conditions and the solitary waves 
have been the same in all cases. 

\begin{figure}[H]
  \centering
  \includegraphics[width=0.7\columnwidth]{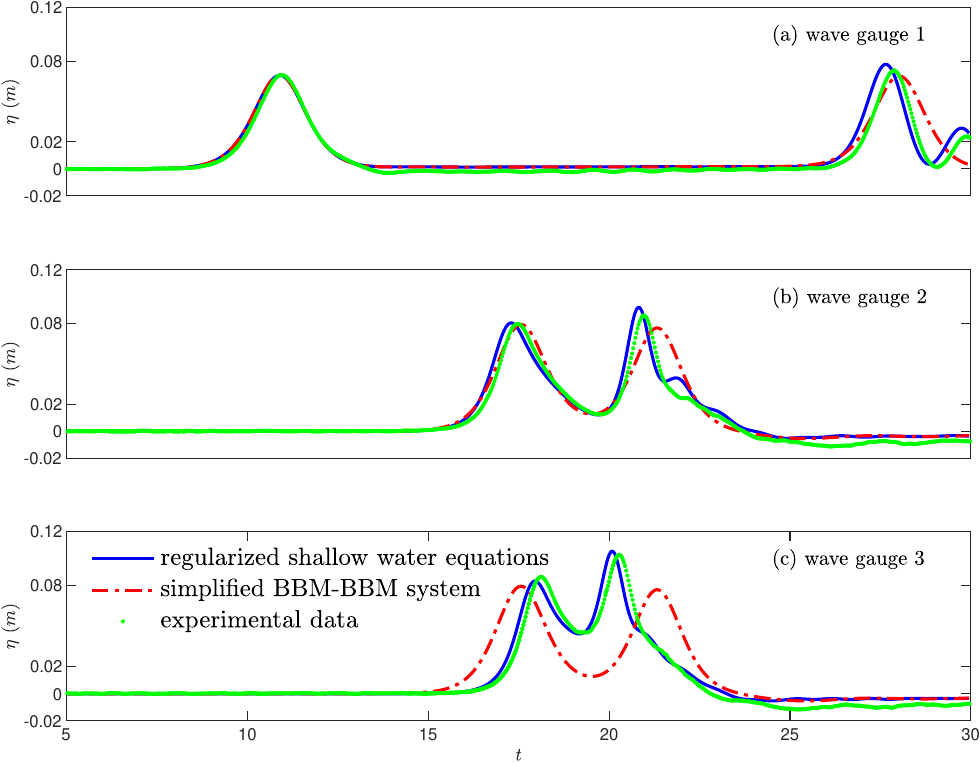}
  \caption{Surface elevation recorded at the three wave-gauges (A=0.07)}
  \label{fig:f1}
\end{figure}
\begin{figure}[H]
  \centering
  \includegraphics[width=0.7\columnwidth]{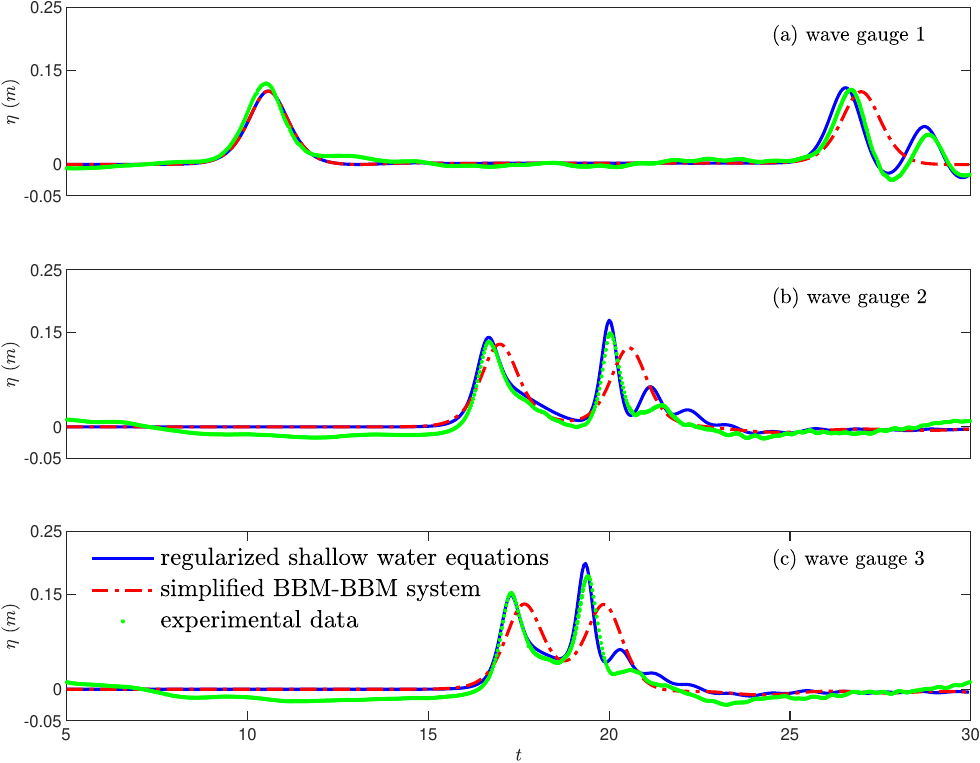}
  \caption{Surface elevation recorded at the three wave-gauges (A=0.12)}
  \label{fig:f2}
\end{figure}
It is worth mentioning that in these two experiments we recorded the integrals
$$M^h(t)=\int_{\Omega} \eta^h ~d \bx, \quad \mbox{ and } \quad E^h(t)=\frac{1}{2} \int_{\Omega} g [\eta^h]^2 + (D+\eta^h)|\bu^h|^2~d\bx \ .$$
In both cases the excess mass $M^h$ was conserved nearly to machine precision. The function $E^h$ was conserved to 5 digits.


\subsection{Interaction of a solitary wave with a cylindrical obstable}

In this section we consider a genuine 2D experiment describing the interaction of a solitary wave with a vertical ellipsoidal cylinder. In particular we consider the propagation of a classical, line solitary wave of amplitude $0.04$ propagating in a channel with horizontal dimensions $[-15,15]\times[-5,5]$ and depth $0.2$ (all the dIstances are in meters). A vertical ellipsoidal cylinder with major axis $2$, minor axis $1$ is located at the center of the channel with its center at the origin. The sketch of the domain is depicted in Figure \ref{fig:figure8}. 
\begin{figure}[H]
  \centering
  \includegraphics[width=0.6\columnwidth]{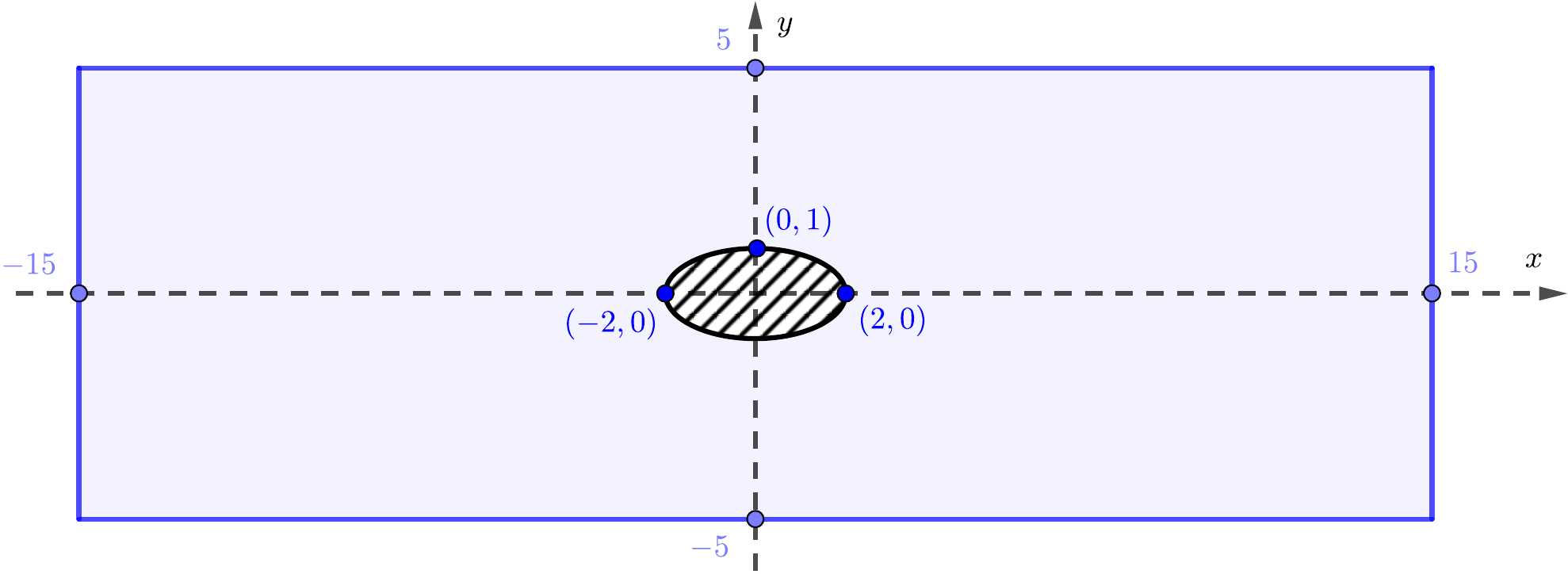}
     \caption{Sketch of the horizontal channel with a vertical ellipsoidal cylinder}
  \label{fig:figure8}
\end{figure}

\begin{figure}[ht!]
  \centering
  \begin{tabular}{cc}
  Classical BBM-BBM system & Regularized shallow water equations \\
  \includegraphics[width=0.48\columnwidth]{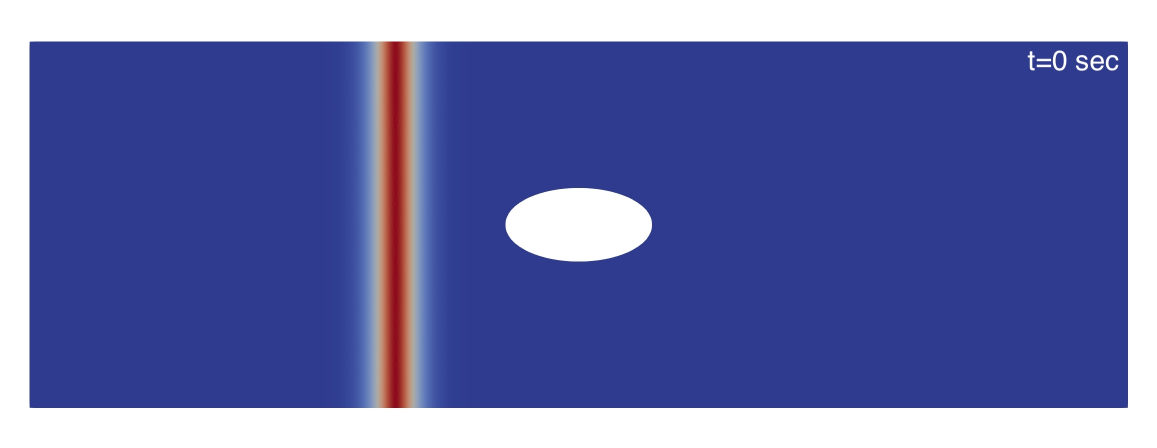} & \includegraphics[width=0.48\columnwidth]{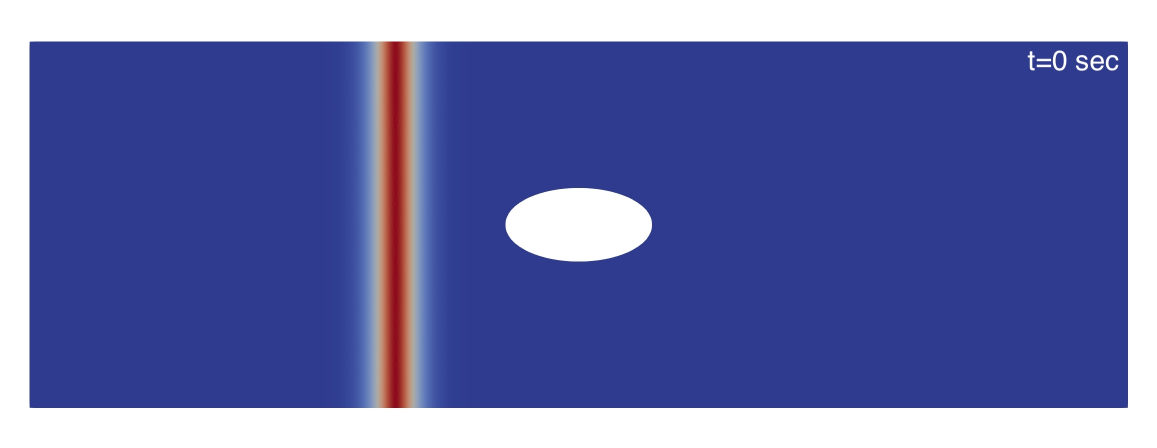} \\
  \includegraphics[width=0.48\columnwidth]{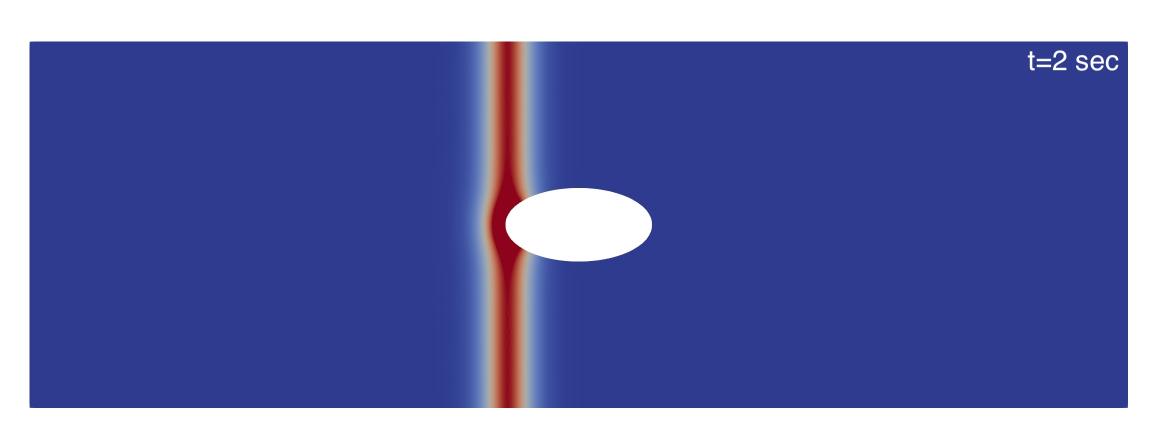} & \includegraphics[width=0.48\columnwidth]{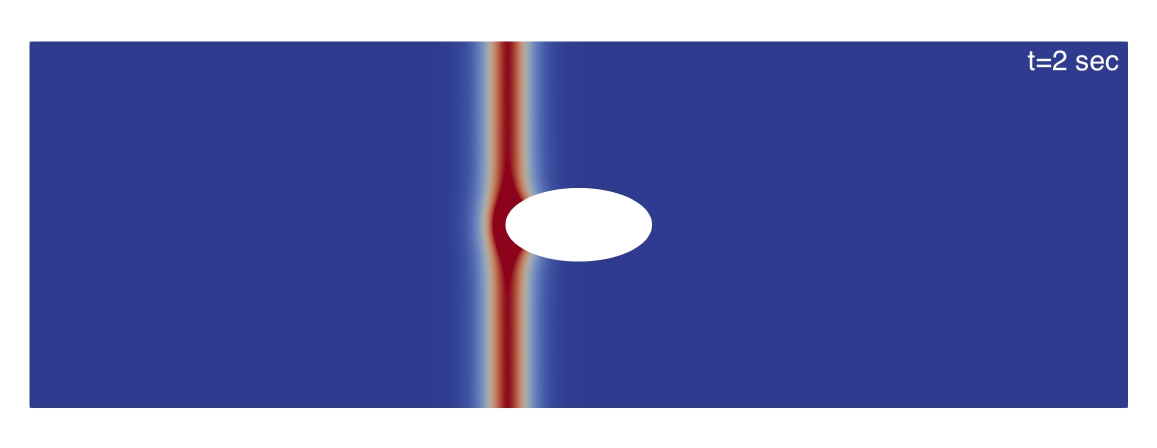} \\
  \includegraphics[width=0.48\columnwidth]{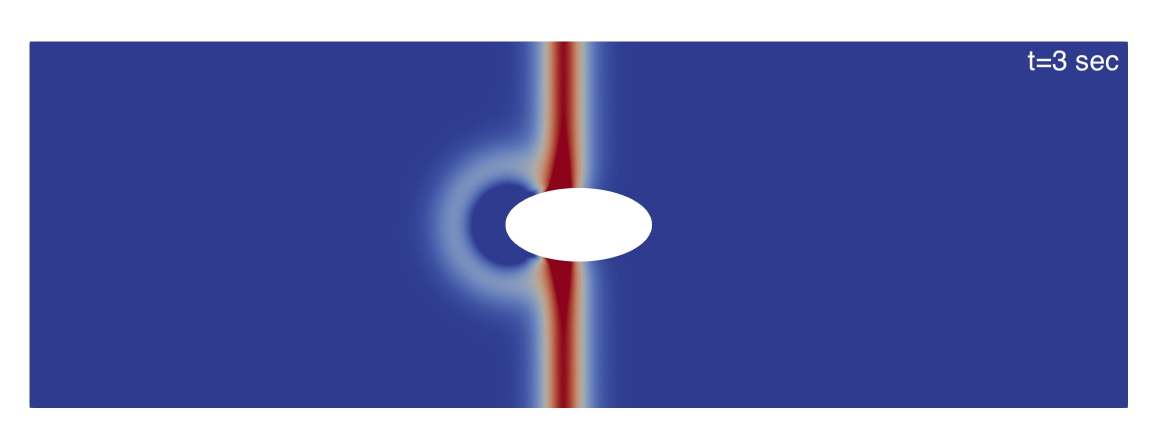} & \includegraphics[width=0.48\columnwidth]{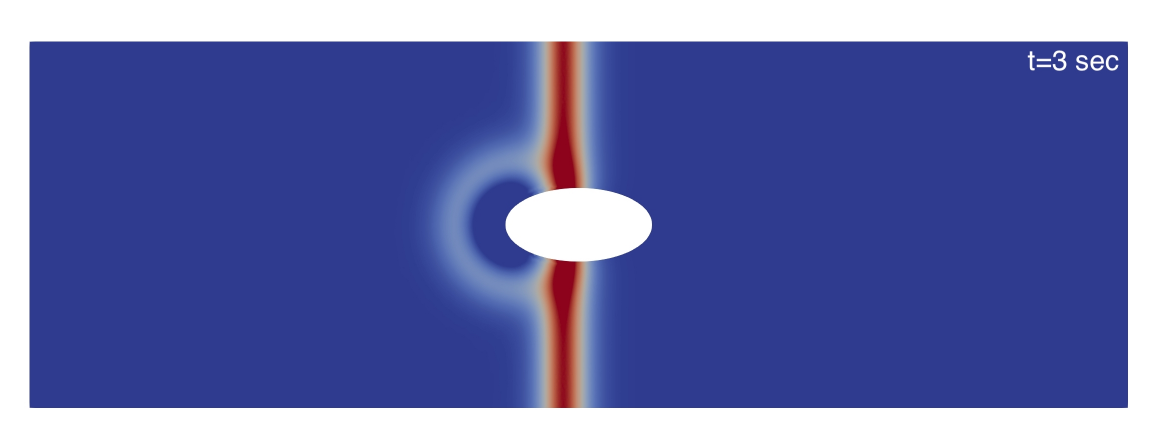} \\
  \includegraphics[width=0.48\columnwidth]{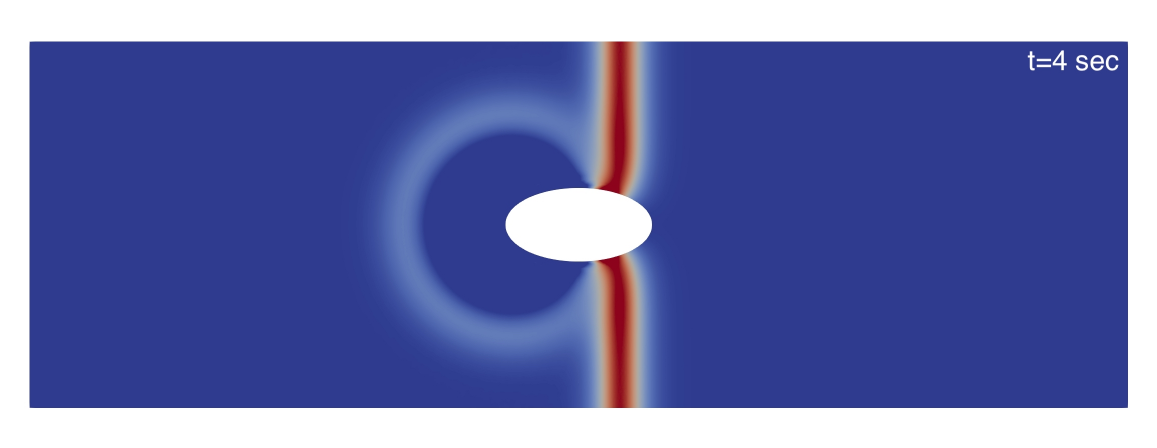} & \includegraphics[width=0.48\columnwidth]{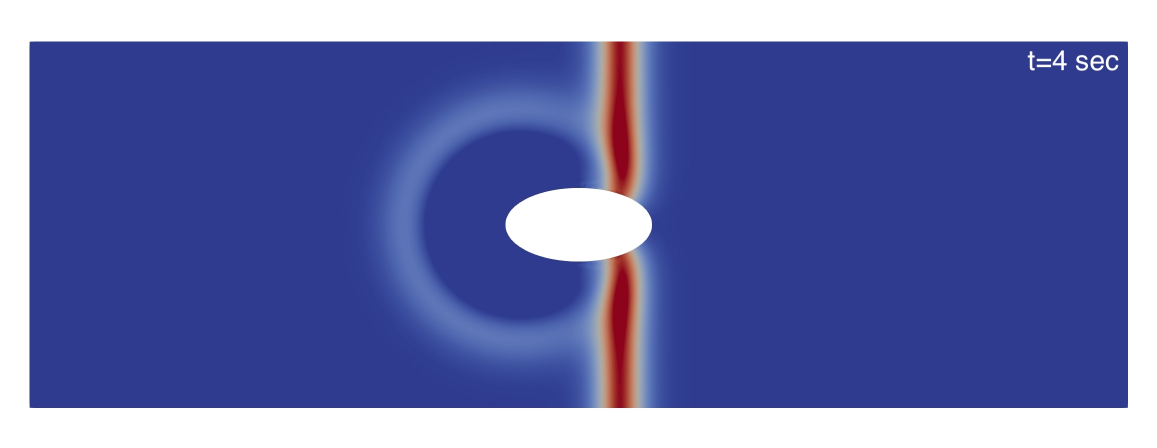} \\
  \includegraphics[width=0.48\columnwidth]{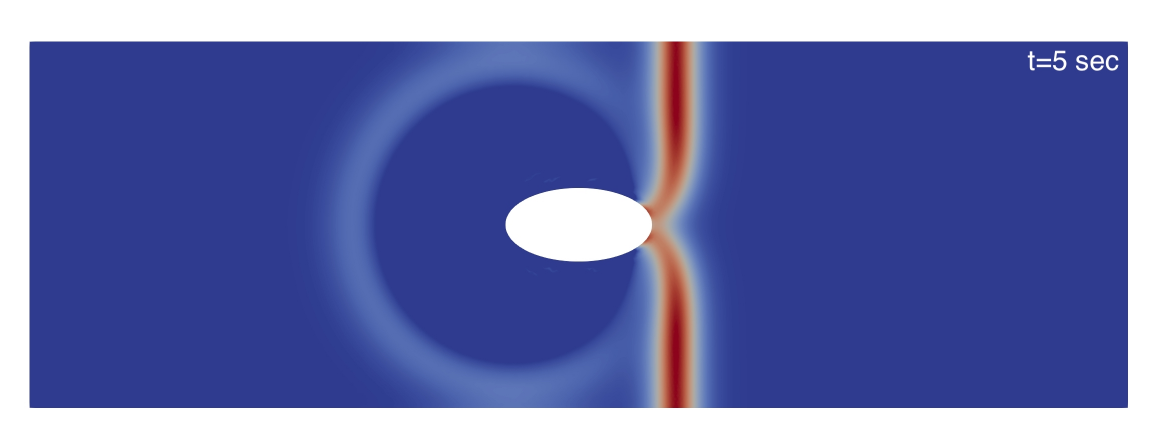} & \includegraphics[width=0.48\columnwidth]{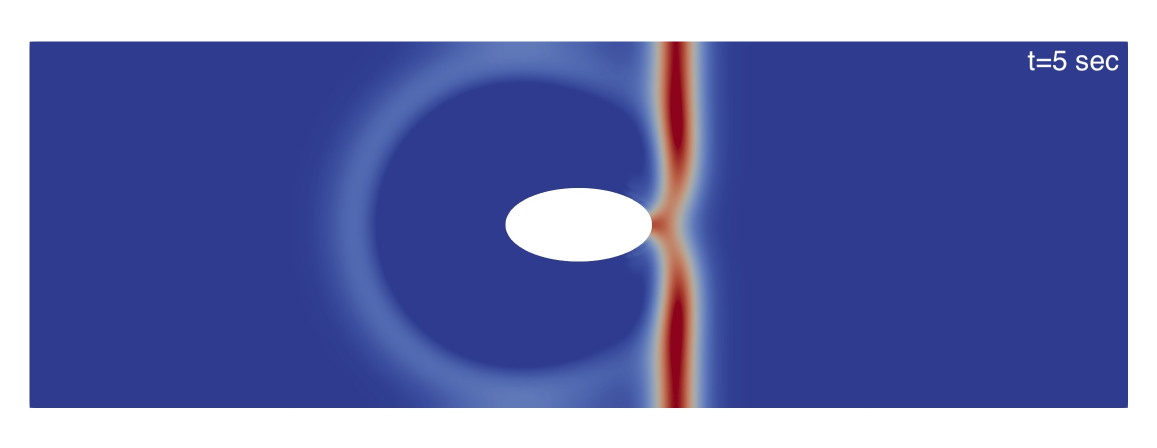} \\
  \includegraphics[width=0.48\columnwidth]{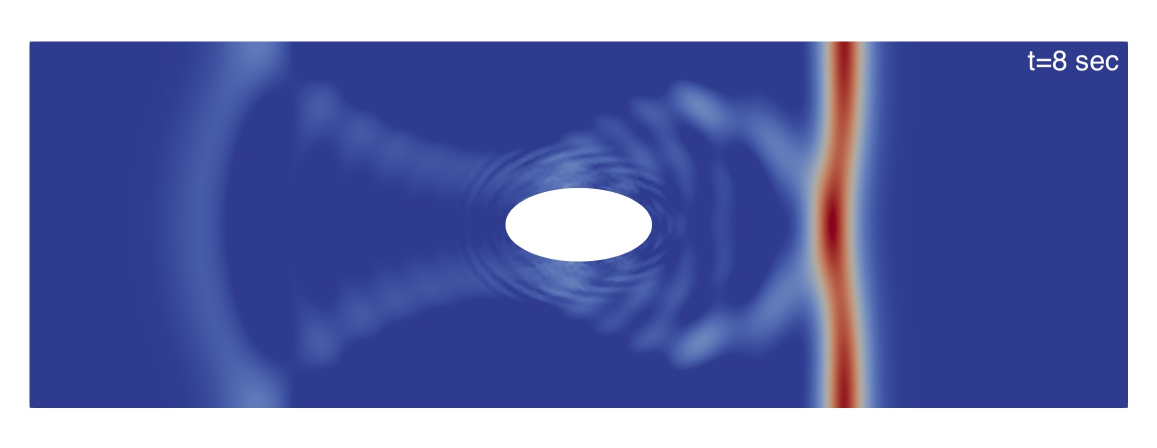} & \includegraphics[width=0.48\columnwidth]{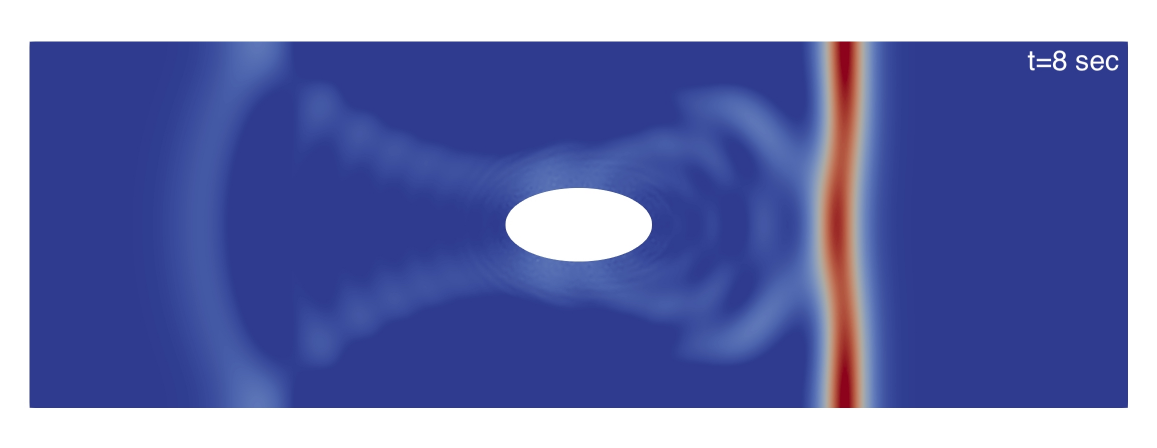} 
  \end{tabular}
  \caption{Top view of the interaction of a solitary wave with a vertical ellipsoidal cylinder. A comparison between the classical BBM-BBM system and the regularized shallow water equations}
  \label{fig:cylinder1}
\end{figure}

The propagation of line solitary waves in the channel requires slip-wall boundary conditions on every part of the boundary. The classical BBM-BBM system (\ref{eq:BBMsys}) is well-posed with no-slip-wall boundary conditions ($\nabla\eta\cdot\bn=0$ and $\bu=0$ on $\partial\Omega$). With such conditions, the line solitary wave sticks to the side-walls and cannot propagate without any change, unless we use for compatibility reasons homogeneous Neumann boundary conditions for both $\eta$ and $\bu$ on the sides of the channel only, \cite{DMS2007}. Such conditions allow line solitary waves to propagate along the channel without change in shape, but they cannot simulate accurately reflections \cite{ADM2009,ADM2010,DMS2007,DMS2009,DMS2010}. In addition to these Neumann boundary conditions for the classical BBM-BBM system we use no-slip-wall boundary conditions on the boundary of the cylinder to compare with the results from the new system. The new regularized shallow water equations can be used naturally with slip-wall boundary conditions applied on every part of the boundary of the domain. (For both systems we take $g=9.81~m/sec^2$.)

For the numerical solution of the classical BBM-BBM system we use the standard Galerkin method which was presented and analyzed in detail in \cite{DMS2007}. For the new system we use the numerical method of Section \ref{sec:numsols}. For the time discretization we employ once again the classical four-stage Runge-Kutta method of order four where we integrate the system until time $T=10~sec$ and with $\Delta t=0.05~sec$. A regular, unstructured mesh of the computational domain with $N_h=72,652$ triangles is considered with $(r,p)=(1,2)$. The common for both systems solitary wave has amplitude $0.04~m$, and is generated numerically using the Petviashvili method of \cite{MRKS2021} adapted appropriately in two dimensions \cite{KMS2019}. During the experiment we record the free surface elevation at three locations (wave gauges): $G_1(-2,0)$, $G_2(0,1)$, $G_3(2,0)$ to measure the runup around the cylinder.

Figure \ref{fig:cylinder1} presents the interaction of the solitary wave with the vertical cylinder. We observe that the slip-wall and no-slip-wall boundary conditions result in different solutions. In particular, we observe that while the slip-wall boundary conditions allow the solitary wave to slide around the obstacle, the no-slip-wall conditions causes a speed reduction of the solitary wave, especially for the parts of the wave close to the cylinder. 
\begin{figure}[H]
  \centering
  \includegraphics[width=0.8\columnwidth]{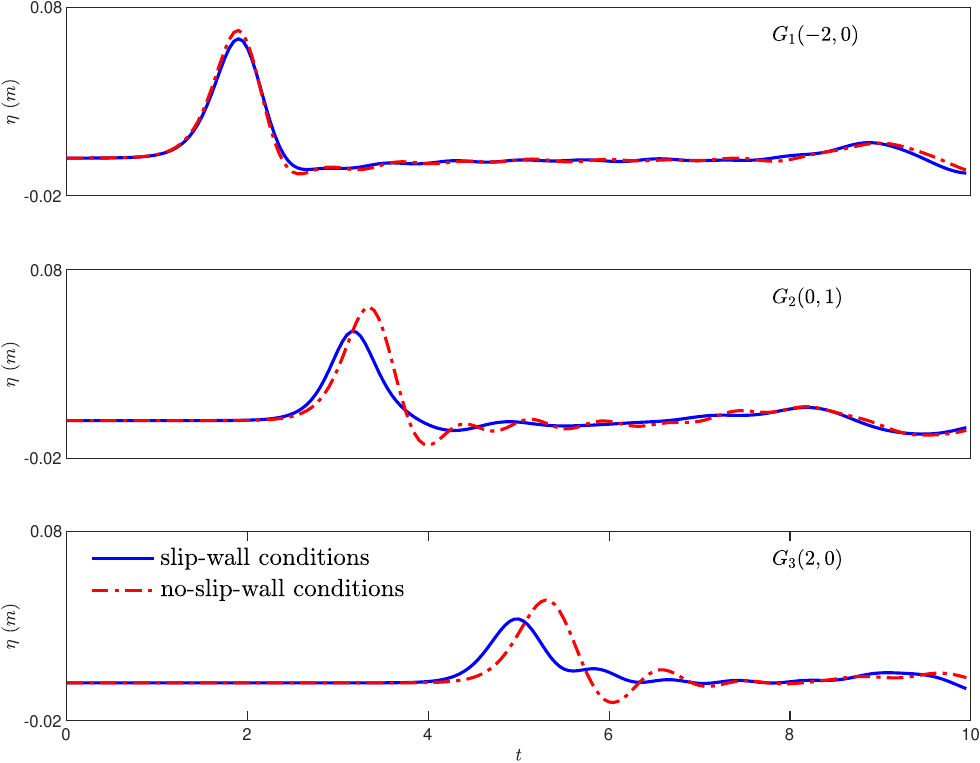}
  \caption{Runup of the solitary wave on vrious locations around the cylinder. Comparison between slip and no-slip conditions}
  \label{fig:runup}
\end{figure}
Figure \ref{fig:runup} presents the recorded values of the solution at the three wave-gauges. The classical BBM-BBM system seems to predict well the runup at the western side of the cylinder. On the other hand, the runup on the north and east sides of the cylinder are not in agreement with the new regularized shallow water system. A delay in the arrival time of the wave is observed due to the no-slip conditions. Considering longer obstacles (longer major axis) one can observe longer delays in the arrival time of the wave on the east side of the cylinder.

\section{Conclusions}

A new Boussinesq system of BBM-BBM type for the propagation of small-amplitude 
long waves has been derived under the smooth bottom variations assumption. 
The new system is appropriate for the study of waves in bounded domains with 
smooth boundary using slip-wall boundary conditions. 
The well-posedness of the specific initial-boundary value problem of the new system was established 
in appropriate Sobolev spaces. Furthermore, a Galerkin / Finite element method was used for the 
semi-discretization of its weak formulation. Nitsche's method for the implicit imposition 
of the boundary conditions was used. The semi-discretization was analyzed theoretically 
by proving the convergence and estimating the errors in appropriate norms. 
The theoretical findings were also validated in practice using appropriate experiments,
and good agreement was found.


\bibliographystyle{plain}
\bibliography{biblio}
\bigskip

\end{document}